\documentclass[11pt]{amsart}

\usepackage[margin=1.1in]{geometry}
\usepackage{math_headers}
\usepackage{mathtools}
\usepackage{bm}

\newcommand{\PiPop}{\Pi^{\mr{pop}}_{n,\beta}}
\newcommand{\PiEmp}{\Pi^{\mr{emp}}_{n,\beta}}
\newcommand{\sPiPop}{\smash{\Pi^{\mr{pop}}_{n,\beta}}}
\newcommand{\sPiEmp}{\smash{\Pi^{\mr{emp}}_{n,\beta}}}

\title{Linear Response Estimators for Singular Statistical Models}
\author{Chris Elliott and Daniel Murfet}
\date{\today}

\begin{document}

\begin{abstract}
We define \emph{susceptibilities} as a measure of the response of an observable quantity of a parameterized statistical model to a perturbation of the data for a general class of observables.  We define estimators for these susceptibilities as statistics in a sequence of $n$ data-points and prove that these estimators are consistent and asymptotically unbiased in the large $n$ regime.
\end{abstract}

\maketitle

\section{Introduction}
\emph{Susceptibilities} \cite{lang1,lang2,lang3} are an interpretability tool by which one may study the response of a parameterized statistical model such as a neural network to a perturbation of its environment.  At the theoretical level these susceptibilities are well-grounded by the field of singular learning theory, in which the (Bayesian) learning of a statistical model is controlled by the singular geometry of the loss.  A closely related object is the \emph{Bayesian influence function} of \cite{singfluence1}, which measures the sensitivity of posterior expectations to reweighting individual training samples.  In reality one uses a \emph{susceptibility estimator}, which is a statistic in the data, in order to apply the conceptual interpretation of susceptibilities in real examples.

Susceptibilities depend on a choice of two parameters: a perturbation of the data distribution, and an observable -- a generalized function on the parameter space.  For any perturbation and observable satisfying mild conditions we construct estimators for the associated population level susceptibility, depending on a sequence of data-points $x_1, \ldots, x_n$, and show that these statistics are \emph{consistent}.  A precise statement of our main result is as follows.  In the statement below we fix a data manifold $X$ with true probability distribution $q$, a real analytic family of probability distributions $p_w$ on $X$ parameterized by $w \in W$ and study the Kullback--Leibler loss function
\[K(w) = \int_X q(x)f(x,w) \d x = \int_X q(x) \log \frac{q(x)}{p_w(x)} \d x.\]

\begin{theorem} \label{main_theorem}
Let $\OO = P \delta_S$ where $P$ is a linear differential operator of order $M$ and $S \sub W$ is a neatly embedded submanifold of parameter space.   Suppose that all coefficient lifts of $P$ are $L^4(q)$-valued real analytic, and that all coefficient lifts and their derivatives up to order $M$ have relatively finite variance.  Let $\xi \in L^4_0(q)$ be a perturbation and write $\Delta K(w) = \int_X \xi(x) f(x,w) q(x) \d x$ for the loss variation.  Then as $n\beta \to \infty$ and $\beta \to 0$,
\[\mr{Cov}^{\mr{res}}_{\sPiEmp}(\OO_n, \Delta K_n) - \mr{Cov}^{\mr{res}}_{\sPiPop}(\OO, \Delta K) \to 0\]
in probability, where the restricted covariance for differential observables includes the normalization factor $(n\beta)^{-M}$ (see \S\ref{sec:differential}).
\end{theorem}

The proof proceeds by reducing differential observables to functional (zeroth-order) ones via integration by parts and the Leibniz rule, then establishing convergence in the zeroth-order case using a coupling kernel that generalizes the loss kernel of \cite{sing2}. For precise statements and proofs we refer to \S\ref{sec:functional} (Theorem \ref{convergence_zeroth_order}) for the functional case and \S\ref{sec:differential} (Theorem \ref{convergence_differential}) for the general case.

To orient the reader we briefly describe the objects that appear in the paper and the relationships between them.  The \emph{population susceptibility} $\chi^{\mr{pop}}_{n,\beta}(\OO, \xi)$ is defined as a derivative of an expectation under the ``annealed'' posterior distribution with density function proportional to $\exp(-n\beta K(w)) \varphi(w)$; in Lemma \ref{susceptibility_covariance_lemma} we identify this derivative with a covariance against the perturbation of the loss.  In practice one cannot compute this object: the population loss $K(w)$ requires knowledge of the true distribution $q$, and the observable $\OO$ may involve coefficients that are themselves population expectations.  Instead, the quantities estimated in \cite{lang2,lang3} are covariances computed under the actual tempered posterior with density proportional to $\exp(-n\beta K_n(w)) \varphi(w)$, with all population coefficients replaced by empirical estimators for the associated observables.  The main result of this paper (Theorem \ref{main_theorem}) establishes that this empirical quantity converges to its population counterpart in the SLT regime $n\beta \to \infty$, $\beta \to 0$.  Under a mild additional integrability condition (Theorem \ref{asymptotic_unbiasedness}) the estimator is also asymptotically unbiased.  A further layer of approximation arises in practice where one must use Monte Carlo techniques like stochastic gradient Langevin dynamics to estimate sampling from the posterior; this introduces an additional source of potential bias that we do not attempt to address in the present work.

\subsection{Related work}

The identification of a susceptibility with a posterior covariance has a long history.  In the theory of robust statistics, the \emph{influence function} of Hampel \cite{hampel1974influence} measures the sensitivity of a statistical functional to infinitesimal contamination of the data distribution.  In the Bayesian setting, Gustafson \cite{gustafson2000local} studies the derivative of posterior expectations with respect to perturbations of the prior.  Giordano, Broderick, and Jordan \cite{giordano2018covariances} use this classical identity to correct the covariance estimates produced by mean-field variational Bayes, which systematically underestimates posterior variances.  Koh and Liang \cite{koh2017understanding} revived influence functions for deep learning, using Hessian-vector products to attribute predictions to training points.  The \emph{Bayesian influence function} of \cite{singfluence1} replaces the Hessian inversion with posterior covariance estimation via SGLD, and Mlodozeniec et al.\ \cite{mlodozeniec2025distributional} prove a related consistency result in a distributional training data attribution framework.  The susceptibility of \cite{lang2} is a special case of this family of ideas.  In the present paper we prove that the empirical susceptibility estimator is consistent in the singular setting, where the Fisher information degenerates and the Hessian-based approach is ill-defined.

Our main theorem operates in the regime $n\beta \to \infty$, $\beta \to 0$, where $\beta$ is the inverse temperature parameter appearing in the tempered posterior $\propto \exp(-n\beta K(w))\varphi(w)$.  This double limit is central to singular learning theory: Watanabe's widely applicable Bayesian information criterion (WBIC) \cite{WatanabeWBIC} uses $\beta = 1/\log n$ to estimate the free energy of a singular model, and the asymptotic expansions of \cite{WatanabeGrey, WatanabeGreen} on which we rely operate in this regime.  Ray, Avella Medina, and Rush \cite{ray2026tempering} study the same double limit for regular models, proving that all posterior moments converge to Gaussian moments under locally asymptotically normal (LAN) conditions, with a sharp threshold at $\beta \sim 1/\sqrt{n}$ below which the Bernstein--von Mises theorem fails.  Our results complement theirs by establishing convergence of a covariance functional in the singular setting, where the LAN condition does not hold.  See also Drton and Plummer \cite{drton2017sbic} and Bochkina and Green \cite{bochkina2014nonregular}.

The term \emph{susceptibility} originates in statistical mechanics, where the static susceptibility is related to an equilibrium covariance by the fluctuation-dissipation theorem \cite{Kubo}.  The coupling kernel introduced in this paper generalizes the \emph{loss kernel} of \cite{sing2}, which measures functional coupling between data points through the posterior.

In practice one may approximate posterior samples using stochastic gradient Langevin dynamics \cite{wellingteh2011}.  The discretization bias introduced by SGLD is analyzed by Vollmer, Zygalakis, and Teh \cite{vollmer2016sgld}, who show that at constant step size the invariant distribution has non-vanishing bias.  Standard global convergence guarantees for SGLD rely on assumptions that are likely incompatible with the degenerate loss landscapes of singular models; Hitchcock and Hoogland \cite{hitchcock2025global} argue for a shift of focus to local posterior sampling and provide empirical benchmarks.  The biases introduced by this additional layer of approximation (see Remark \ref{rmk:three_layers} and \S\ref{SGLD_section}) are separate from the population-to-empirical bridge established here and are not investigated in the present work.

\section{Setup}

We begin by laying out the class of statistical models to which our results apply.  We will consider the setting of \emph{singular models} as defined by Watanabe \cite{WatanabeGrey,WatanabeGreen}.  That is, we study a distribution learning setting where the loss function is not required to be Morse, or even Morse--Bott: in particular its minimal locus does not typically consist of isolated points, but instead may consist of a \emph{singular subspace} of positive dimension and potentially elaborate geometric structure.

Let $W \sub \RR^d$, our \emph{parameter space}, be a compact subanalytic set determined by finitely many real analytic inequalities $\pi_1(w) \ge 0, \ldots, \pi_k(w) \ge 0$, equipped with the Lebesgue measure $\d w$.  Fix a real analytic family $p(\cdot|w) \d x$ of probability distributions on a data manifold $X$ with volume form $\d x$ parameterized by $w \in W$, a \emph{true} distribution $q(x) \d x$ on $X$ with strictly positive smooth density, and a real analytic \emph{prior} probability distribution $\varphi(w) \d w$ on $W$.  

\begin{remark}
In the language of differential geometry $W$ is a real analytic \emph{manifold with corners} \cite{Joyce, Melrose}.  When we talk about local \emph{charts} in $W$ we will mean charts in this setting, i.e. open subsets $U \sub W$ equipped with real analytic diffeomorphisms 
\[\psi \colon U \to \RR_{\ge 0}^j \times \RR^{k-j}\]
for some pair of natural numbers $j \le k$.
\end{remark}

Following \cite{WatanabeGrey}, write
\[f(x,w) = \log \frac{q(x)}{p(x|w)}\]
for the log density ratio.  
\begin{definition}
The \emph{population loss} is the Kullback--Leibler function $K \colon W \to \RR$ defined by
\[K(w) = \mr{KL}(q || p_w) = \int_X f(x,w) q(x) \d x.\]
Let $n$ be a natural number and let $x_1, \ldots, x_n$ be i.i.d random variables sampled from the true distribution $q$. The \emph{empirical loss} is
\[ K_n(w) = \frac 1n \sum_{i=1}^n f(x_i, w).\]
\end{definition}

The empirical loss $K_n$ is a random real analytic function on $W$.  We will assume it satisfies an integrability condition in the data direction in the following sense (see also the discussion in \cite[Section 5.1]{WatanabeGrey}).
\begin{definition}
Let $s \in [1,\infty]$.  A function $F \colon X \times W \to \RR$ is \emph{$L^s(q)$-valued real analytic} if for all $w \in W$ there exists an open chart $U$ containing $w$ with local coordinate $u$ for which
\[F(x,u) = \sum_{\alpha} a_\alpha(x) u^\alpha\]
is a power series with $a_\alpha \in L^s(q)$ for all $\alpha$, absolutely convergent in the normed space $L^s(q)$ for $u$ in a disk of positive radius. The sum here ranges over multi-indices $\alpha \in \ZZ_{\ge 0}^d$.
\end{definition}

\begin{assumption} \label{standing_assumption}
We impose the standing assumption that $f$ is $L^4(q)$-valued real analytic.
\end{assumption}

We will often wish to study derivatives of $L^s(q)$-valued real analytic functions in the $W$ direction.  We will use the following fact.
\begin{prop} \label{derivative_of_Ls_real_analytic}
If $F$ is $L^s(q)$-valued real analytic then so is any partial derivative $\dd^\beta_w F$ for $\beta \in \ZZ_{\ge 0}^d$.
\end{prop}

We refer to \cite[Lemma 4.2.7]{HitchcockThesis} for a proof.  We also invoke the following implication from $L^s(q)$-valued to pointwise analyticity when applying moment lemmas to insertions and when verifying domination conditions.

\begin{prop} \label{Ls_pointwise}
Let $F$ be $L^s(q)$-valued real analytic on a compact subspace $W \sub \RR^d$.  There is a measurable representative of $F$ for which:
\begin{enumerate}
    \item for $q$-almost every $x \in X$, the function $w \mapsto F(x, w)$ is real analytic on $W$;
    \item $\sup_{w \in W} |F(\cdot, w)| \in L^s(q)$.
\end{enumerate}
\end{prop}

\begin{proof}
Cover $W$ by finitely many open polydiscs $U_k$ inside which the local Taylor series $F(w) = \sum_\alpha a^{(k)}_\alpha (w - w_k)^\alpha$ converges absolutely in $L^s(q)$.  Fix a radius $r_k$ strictly inside the radius of convergence of the $k$-th polydisc and set $G_k(x) := \sum_\alpha |a^{(k)}_\alpha(x)| r_k^{|\alpha|}$.  By the triangle inequality $\|G_k\|_{s} \le \sum_\alpha \|a^{(k)}_\alpha\|_{s} r_k^{|\alpha|} < \infty$, so $G_k \in L^s(q)$ and is finite for $q$-almost every $x$.  Hence the pointwise series converges absolutely on the polydisc of radius $r_k$ for $q$-almost all $x$, defining a real analytic function of $w$; gluing across the finite union of null sets gives us (1).  For (2), $\sup_{w \in U_k} |F(x, w)| \le G_k(x)$ on each $U_k$, and the supremum over the compact $W$ is dominated by $\max_k G_k \le \sum_k G_k \in L^s(q)$.
\end{proof}

We will introduce the notion following \cite{lang2} of a \emph{susceptibility} as an instantiation in the singular learning setting of linear response theory.  The term is borrowed from statistical mechanics (as in e.g. \cite{Kubo}), where the susceptibility of a system is the derivative of a response quantity with respect to an external field.  In the Bayesian robustness literature the same concept is referred to as \emph{sensitivity}: a derivative of a posterior functional with respect to a perturbation of the prior or data distribution \cite{gustafson2000local}.  To define these data we will need to specify two inputs.
\begin{enumerate}
    \item A \emph{perturbation} of the true distribution $q$, i.e. an infinitesimal linear deformation of $q$ in the space of probability distributions.
    \item An \emph{observable} on $W$, i.e. a generalized function whose expectation we can study the variation in under the perturbation.
\end{enumerate}
Setting up both of these inputs carefully requires some assumptions, which we now spell out.

\begin{definition}
A \emph{perturbation} of $q$ is a tangent vector $\tau \in T_q \Delta^+(X)$, where $\Delta^+(X)$ is the space of strictly positive smooth probability density functions given its usual Fr\'echet manifold structure.  We identify such perturbations with signed density functions on $X$ and restrict to $q$-absolutely continuous measures of the form $\tau = \xi q$, where $\xi \in L^4_0(q) \cap C^\infty(X) := \{\xi \in L^4(q) \cap C^\infty(X) : \int_X \xi(x) q(x) \d x = 0\}$.
\end{definition}

\begin{remark} \label{rmk:perturbation_regularity}
The full tangent space $T_q \Delta^+(X)$ is canonically identified, via $h \mapsto h/q$, with the Fr\'echet space $C^\infty_0(q) = \{\xi \in C^\infty(X) : \int \xi\, q = 0\}$ of smooth functions with zero $q$-expected value. We are restricting attention to perturbations that satisfy an additional finiteness condition on their fourth moment.

 The $L^4(q)$ integrability condition is used in two ways.  First, it ensures that the loss variation $\Delta K(w) = \int \xi(x) f(x,w) q(x) \d x$ is well-defined and that the product $\xi f$ is $L^2(q)$-valued real analytic in $w$ (by H\"older's inequality, since $f$ is $L^4(q)$-valued real analytic by Assumption \ref{standing_assumption}).  This is the regularity required to apply Watanabe's empirical process bound \cite[Theorem 5.8]{WatanabeGrey} to the deformation of the loss under $\tau$.  Second, it ensures that the domination condition needed in Remark \ref{rmk:domination} is satisfied.  More general integrability conditions are discussed in Remark \ref{rmk:holder_tradeoff}.
\end{remark}

We will now introduce the other input into our linear response framework, an \emph{observable} on the parameter space $W$.  We referred to observables above as ``generalized functions''.  One can make this precise by talking about, for instance, the space of (Schwartz) distributions on $W$.  We will not wish, however, to work at this level of generality.  Instead we wish to restrict attention to those distributions generated by derivatives of delta distributions on sufficiently well-behaved submanifolds $S \sub W$.  We make this precise in the following way.

First, if $S \sub W$ is a compact subanalytic space we consider the $\delta$-distribution on $S$, denoted $\delta_S$, by
\[\delta_S(\Phi) = \int_S \Phi \d s\]
where $\Phi \in C^\infty(W)$ is a test function and $\d s$ is the restriction of the volume form on $W$ to $S$.  We will characterize ``well-behaved'' submanifolds according to the following definition.

\begin{definition}
An inclusion $\iota \colon S \inj W$ of subanalytic spaces in $\RR^d$ is a \emph{neat embedding} if $\iota$ is an embedding, $\dd S = S \cap \dd W$ and $S$ meets each boundary stratum of $W$ transversally.
\end{definition}

\begin{definition}
An \emph{observable} on $W$ is a distribution $\OO$ that admits a finite decomposition
\[\OO = \sum_i P_i \delta_{S_i}\]
where each $S_i \sub W$ is a closed subanalytic space neatly embedded in $W$ and each $P_i = \sum_{|\alpha| \le M_i} g_{i,\alpha} \dd^\alpha$ is a linear differential operator of order $M_i$ in the ambient coordinates of $\RR^d$.  Each coefficient $g_{i,\alpha}(w) = \int_X \wt g_{i,\alpha}(x,w) q(x) \d x$ is the population expectation of an $L^4(q)$-valued real analytic function $\wt g_{i,\alpha}$ on $W$ (in the sense of \cite[\S5.2]{WatanabeGrey}).  We call $\wt g_{i,\alpha}$ the \emph{lift} of $g_{i,\alpha}$.

An observable is \emph{functional} (or \emph{zeroth-order}) if $M_i = 0$ for all $i$ (i.e.\ each $P_i$ is multiplication by $g_i$), and \emph{differential} otherwise.

An observable is \emph{deterministic} if every coefficient lift is independent of $x$, i.e.\ $\wt g_{i,\alpha}(x,w) = g_{i,\alpha}(w)$ for all $x$.

Given data $x_1, \ldots, x_n$, the \emph{empirical observable} $\OO_n$ is obtained by replacing each coefficient $g_{i,\alpha}$ by its empirical counterpart $g_{i,\alpha,n}(w) = \frac 1n \sum_{j=1}^n \wt g_{i,\alpha}(x_j, w)$.
\end{definition}

\begin{example}
    The component observable $\phi_C(w) = \delta(u - u^*)[K(w) - K(w^*)]$ used in \cite{lang1,lang2} has $S = \{u^*\} \times C$, $\dd^{\alpha} = 1$ (zeroth order), and coefficient $g(w) = K(w) - K(w^*) = \int_X [f(x,w) - f(x,w^*)] q(x) \d x$ with lift $\wt g(x,w) = f(x,w) - f(x,w^*)$.  The delta function is deterministic (a geometric constraint on the support) but the coefficient is not.
\end{example}

Having set up our ingredients we will start to lay out the objects from distribution learning theory that we will study in our paper: the \emph{tempered Bayesian posterior distribution} and its population level counterpart.

\begin{definition}
    Fix an inverse temperature parameter $\beta > 0$.  Define the \emph{population and empirical posteriors} as
    \begin{align*}
        \PiPop &\propto \exp(-n\beta K(w)) \varphi(w) \d w \\
        \text{and }  \PiEmp &\propto \exp(-n\beta K_n(w)) \varphi(w) \d w
    \end{align*}
    respectively.  Write $Z^{\mr{pop}}_{n,\beta} = \int_W \exp(-n\beta K(w)) \varphi(w) \d w$ and $Z^{\mr{emp}}_{n,\beta} = \int_W \exp(-n\beta K_n(w)) \varphi(w) \d w$ for the partition functions.  Given an observable $\OO$ we define expectation values as
    \begin{align*}
        \bb E_{\sPiPop}(\OO) &= \frac{1}{Z^{\mr{pop}}_{n,\beta}} \langle \OO, \exp(-n\beta K(w)) \varphi(w) \rangle \\
        \text{and } \bb E_{\sPiEmp}(\OO) &= \frac{1}{Z^{\mr{emp}}_{n,\beta}} \langle \OO, \exp(-n\beta K_n(w)) \varphi(w) \rangle.
    \end{align*}
\end{definition}

\begin{remark}
    The population posterior distribution $\PiPop$ is sometimes, by analogy with statistical mechanics, referred to as the \emph{annealed posterior}.  It only depends on the combined parameter $n\beta$, not on $n$ and $\beta$ individually.
\end{remark}

\begin{definition}
    By linearity it suffices to define the susceptibility for an observable $\OO = P \delta_S$ where $P$ is a linear differential operator of pure order $M$.  The \emph{population susceptibility} of $\OO$ with respect to the perturbation $\xi$ is
    \[\chi^{\mr{pop}}_{n,\beta}(\OO, \xi) = \frac{1}{(n\beta)^{M+1}} \nabla_\xi \bb E_{\sPiPop}(\OO).\]
    For a functional observable ($M = 0$) this reduces to $\frac{1}{n\beta} \nabla_\xi \bb E_{\sPiPop}(\OO)$.
\end{definition}

\begin{remark} \label{rmk:normalization}
The additional factor of $(n\beta)^{-M}$ in the differential case is necessary to obtain a well-defined limit as $n\beta \to \infty$.  Each normal derivative of the posterior density $e^{-n\beta K}\varphi$ brings down a factor of $n\beta$ from the chain rule, so the expectation $\bb E_{\sPiPop}(\OO)$ can grow as $(n\beta)^M$ and without this normalization the susceptibility diverges.  The rescaling ensures that only the leading Fa\`a di Bruno term (involving products of first normal derivatives of $K$, which vanish on the vanishing locus $W_0$ of $K$) contributes in the limit, while sub-leading terms involving higher normal derivatives of $K$ are suppressed by negative powers of $n\beta$.
\end{remark}

We can realize susceptibilities as covariances.  This is a standard identity: in the variational Bayes setting the derivative of a posterior mean with respect to a data perturbation equals a posterior covariance \cite{giordano2018covariances}; we record the version for the tempered posteriors used in singular learning theory.  For sufficiently small $h \in \RR$ let $q^{h\xi} = (1+h\xi)q$ denote the time-$h$ flow of $q$ under $\xi$, let $K^{h\xi}$ denote the corresponding population loss, and write $\Delta K = K^{h\xi} - K$ for the loss variation at scale $h$.  Without loss of generality we may rescale $\xi$ so that we may take $h=1$.

\begin{remark}
One should remark that, given $\xi$, there may not exist an $h>1$ so that the flow $q^{h\xi}$ exists in the Fr\'echet manifold $\Delta^+(X)$.  Indeed, in the Fr\'echet manifold setting the Picard-Lindel\"of theorem may not hold, and vector fields need not be even locally integrable (see \cite[\S 32.12]{KrieglMichor} for a discussion of this point).  One can always find a non-linear integral curve whose derivative at $q$ coincides with the tangent vector $\xi q$.  More straightforwardly, the flow $q^{h\xi}$ does always exist in the space of smooth signed density function and this more general setting is sufficient for the argument that follows.
\end{remark}

\begin{lemma}[Covariance lemma] \label{susceptibility_covariance_lemma}
We may identify the susceptibility associated to the observable $\OO$ and the perturbation $\xi$ as
    \[\chi^{\mr{pop}}_{n,\beta}(\OO, \xi) = - \frac{1}{(n\beta)^M}\mr{Cov}_{\sPiPop}(\OO, \Delta K).\]
\end{lemma}

\begin{proof}
Write $N(h) = \langle \OO, e^{-n\beta K^{h\xi}} \varphi \rangle$ and $Z(h) = \int_W e^{-n\beta K^{h\xi}} \varphi \d w$, so that $\bb E_{\Pi^{h\xi}}(\OO) = N(h)/Z(h)$.  Since $K^{h\xi}(w) = K(w) + h\int \xi(x) f(x,w) q(x) \d x + C(h)$ where $C(h)$ is independent of $w$, we have
\[\frac{\d}{\d h}\bigg|_{h=0} e^{-n\beta K^{h\xi}(w)}
= -n\beta \Big(\int \xi(x) f(x,w) q(x) \d x\Big) e^{-n\beta K(w)}.\]
Differentiating $N(h)/Z(h)$ by the quotient rule and evaluating at $h = 0$:
\begin{align*}
\nabla_\xi \bb E_{\sPiPop}(\OO)
&= -n\beta \Big[\bb E_{\sPiPop}\Big(\OO \cdot \textstyle\int \xi f q \d x\Big)
- \bb E_{\sPiPop}(\OO) \bb E_{\sPiPop}\Big(\textstyle\int \xi f q \d x\Big)\Big] \\
&= -n\beta \mr{Cov}_{\sPiPop}\Big(\OO, \textstyle\int \xi(x) f(x,\cdot) q(x) \d x\Big).
\end{align*}
Dividing by $(n\beta)^{M+1}$ gives
\[\chi^{\mr{pop}}_{n,\beta}(\OO,\xi) = -\frac{1}{(n\beta)^M}\mr{Cov}_{\sPiPop}\Big(\OO, \int \xi f q \d x\Big).\]
Finally, $\Delta K(w) = \int \xi(x) f(x,w) q(x) \d x$ up to $w$-independent terms.
\end{proof}

\begin{remark}
    The function $\chi^{\mr{pop}}_{n,\beta}(\OO)$ sending $\xi$ to $\chi^{\mr{pop}}_{n,\beta}(\OO,\xi)$ is linear and continuous, so defines a covector in the $L^2$ cotangent space $T_{q, L^2}^* \Delta(X)$.  We refer to it as the \emph{response covector}.
\end{remark}

\begin{remark}
Since $\Delta K(w) = \int \xi(x) f(x,w) q(x) \d x$ up to $w$-independent terms, and the covariance is linear in its second argument, the susceptibility decomposes as
\[\chi^{\mr{pop}}_{n,\beta}(\OO, \xi) = \int_X \xi(x) \chi_x(\OO) q(x) \d x\]
where the \emph{per-sample susceptibility} $\chi_x(\OO) = -\frac{1}{(n\beta)^M}\mr{Cov}_{\sPiPop}(\OO, f(x,\cdot) - K(\cdot))$ is the integral kernel of the response covector as a linear functional on $L^4_0(q)$.  The centering by $K$ is a choice of gauge: since $\int \xi q = 0$, adding a function of $w$ alone to the kernel does not change the integral.
\end{remark}

\begin{definition}
For each closed embedded submanifold $S \sub W$, define the \emph{restricted empirical posterior}
\[\Pi^{\mr{emp},S}_{n,\beta} = \frac 1{Z^{\mr{emp},S}_{n,\beta}} e^{-n\beta K_n|_S} \varphi|_S \d \mr{vol}_S, \qquad Z^{\mr{emp},S}_{n,\beta} = \int_S e^{-n\beta K_n|_S} \varphi|_S \d \mr{vol}_S.\]
Define the restricted population posterior $\Pi^{\mr{pop},S}_{n,\beta}$ and its partition function $Z^{\mr{pop},S}_{n,\beta}$ analogously with $K$ in place of $K_n$.
\end{definition}

Given $T$ samples $w_1, \ldots, w_T$ from $\Pi^{\mr{emp},S}_{n,\beta}$, write $\widehat{\bb E}_{S,n}$ for the sample mean and $\widehat{\mr{Cov}}_{S,n}$ for the sample covariance.  Similarly let $\widehat{\bb E}_{W,n}$ denote the sample mean over $T'$ exact samples from the full empirical posterior $\PiEmp$.

\begin{definition} \label{def:restricted}
For a functional observable $\OO = \sum_i g_i \delta_{S_i}$, the \emph{restricted covariance} under $\Pi = \PiPop$ or $\PiEmp$ (and similarly for the restricted posteriors $\Pi^{\mr{pop},S_i}_{n,\beta}$ or $\Pi^{\mr{emp},S_i}_{n,\beta}$) is
\[\mr{Cov}^{\mr{res}}_\Pi(\OO, \Delta K) = \sum_i \Big[\bb E_{\Pi^{S_i}_{n,\beta}}\big(g_i \Delta K\big|_{S_i}\big) - \bb E_{\Pi^{S_i}_{n,\beta}}\big(g_i\big|_{S_i}\big) \cdot \bb E_{\Pi_{n,\beta}}(\Delta K)\Big].\]
The extension to differential observables is given in Section \ref{sec:differential}.
\end{definition}

\section{Functional Observables} \label{sec:functional}

Throughout this section we consider functional observables $\OO = \sum_i g_i \delta_{S_i}$.

\subsection{Renormalized susceptibility and estimators}

\begin{definition}
For a functional observable $\OO = \sum_i g_i \delta_{S_i}$, the \emph{renormalized susceptibility} is
\[\chi^{\mr{pop,ren}}_{n,\beta}(\OO,\xi) = -\mr{Cov}^{\mr{res}}_{\sPiPop}(\OO, \Delta K).\]
\end{definition}

For a functional observable, the full-posterior covariance decomposes as
\[\mr{Cov}_{\Pi_{n,\beta}}(\OO, \Delta K) = \sum_i \frac{Z^{S_i}_{n,\beta}}{Z^W_{n,\beta}} \Big[\bb E_{\Pi^{S_i}_{n,\beta}}\big(g_i \Delta K\big|_{S_i}\big) - \bb E_{\Pi^{S_i}_{n,\beta}}\big(g_i\big|_{S_i}\big) \cdot \bb E_{\Pi_{n,\beta}}(\Delta K)\Big].\]
An ideal estimator would include the empirical partition function ratios $Z^{\mr{emp},S_i}_{n,\beta}/Z^{\mr{emp}}_{n,\beta}$ and target the full susceptibility $\chi^{\mr{pop}}_{n,\beta}(\OO,\xi)$.

In the empirical susceptibility estimator we replace $K$ by $K_n$ in the posterior and each coefficient $g_i$ by its empirical counterpart $g_{i,n}$.  For deterministic coefficients $g_{i,n} = g_i$.

\begin{definition}
    For a functional observable $\OO = \sum_i g_i \delta_{S_i}$, the \emph{ideal susceptibility estimator} is
    \[\widehat\chi_{n,\beta}^{\mr{ideal}}(\OO, \xi) = -\sum_i \frac{Z^{\mr{emp},S_i}_{n,\beta}}{Z^{\mr{emp}}_{n,\beta}} \Big[\widehat{\bb E}_{S_i,n}\big(g_{i,n} \Delta K_n\big|_{S_i}\big) - \widehat{\bb E}_{S_i,n}\big(g_{i,n}\big|_{S_i}\big) \cdot \widehat{\bb E}_{W,n}(\Delta K_n)\Big]\]
    where $\widehat{\bb E}_{S_i,n}$ is the sample mean over $T$ exact draws from $\Pi^{\mr{emp},S_i}_{n,\beta}$, $\widehat{\bb E}_{W,n}$ is the sample mean over $T$ exact draws from $\PiEmp$, and $\Delta K_n$ is the empirical counterpart of $\Delta K$:
    \[\Delta K_n(w) = \frac{1}{n}\sum_{j=1}^n \xi(x_j) f(x_j, w).\]
\end{definition}

In practice the partition function ratios are not estimated.  Following \cite{lang2}, one instead estimates the \emph{restricted} quantity obtained by dropping the prefactors.

\begin{definition}
The \emph{renormalized susceptibility estimator} is
    \[\widehat\chi_{n,\beta}^{\mr{ren}}(\OO, \xi) = -\sum_i \Big[\widehat{\bb E}_{S_i,n}\big(g_{i,n} \Delta K_n\big|_{S_i}\big) - \widehat{\bb E}_{S_i,n}\big(g_{i,n}\big|_{S_i}\big) \cdot \widehat{\bb E}_{W,n}(\Delta K_n)\Big].\]
In the limit $T \to \infty$ the sample means converge to exact posterior expectations.  Note that the renormalized susceptibility depends in general on the chosen decomposition $\OO = \sum_i g_i \delta_{S_i}$.
\end{definition}

The two forms are related by
\begin{equation} \label{eq:pop_vs_ren}
\chi^{\mr{pop}}_{n,\beta}(\OO,\xi) = \sum_i \frac{Z^{\mr{pop},S_i}_{n,\beta}}{Z^{\mr{pop}}_{n,\beta}} \cdot (\text{$i$-th summand of } \chi^{\mr{pop,ren}}_{n,\beta}(\OO,\xi)).
\end{equation}
For observables supported on $W$ itself (i.e.\ $S_i = W$ for all $i$) the two coincide.

\begin{remark} \label{rmk:three_layers}
    The renormalized estimator $\widehat\chi_{n,\beta}^{\mr{ren}}$ involves three layers of approximation relative to $\chi^{\mr{pop,ren}}_{n,\beta}(\OO,\xi)$:
    \begin{enumerate}
        \item replacing the population posterior $\PiPop$ by the empirical posterior $\PiEmp$ (and similarly for restricted posteriors),
        \item replacing coefficients $g_i$ by their empirical counterparts $g_{i,n}$,
        \item replacing posterior expectations by sample means from $T$ draws.
    \end{enumerate}
    In the limit $T \to \infty$, layer (3) vanishes.  Layers (1) and (2) are controlled by the dependence on $n$.
\end{remark}

\begin{remark}
In practice one z-scores the susceptibility estimators before downstream analysis \cite{lang3}, and this standardization removes the partition function ratios of \eqref{eq:pop_vs_ren} exactly when the observables in the family share a common submanifold of support, i.e., when $\OO_1, \ldots, \OO_J$ are observables of the form $\OO_j = g_j \delta_S$ all supported on the same submanifold $S$.  In this case the relation between renormalized and full susceptibilities reduces to a single multiplicative factor $\chi^{\mr{pop}}_{n,\beta}(\OO_j, \xi) = (Z^{\mr{pop},S}_{n,\beta}/Z^{\mr{pop}}_{n,\beta}) \cdot \chi^{\mr{pop,ren}}_{n,\beta}(\OO_j, \xi)$ that is independent of $j$.  This factor appears identically in every entry of the family $\{\chi^{\mr{pop,ren}}_{n,\beta}(\OO_j, \xi)\}_{j}$ and in their standard deviation and so is eliminated in the z-scoring.
\end{remark}

\subsection{SGLD-based estimation} \label{SGLD_section}

In practice we cannot genuinely sample from $\Pi^{\mr{emp},S}_{n,\beta}$.  As a replacement we use stochastic gradient Langevin dynamics (SGLD) \cite{wellingteh2011} to approximate samples from the restricted posterior.  The localized posterior on a submanifold $S \sub W$ centered at $w^* \in S$ is
\[p(w; w^*, \beta, \gamma) \propto \exp\big(-n\beta K_n(w) - \tfrac \gamma 2 \|w - w^*\|^2\big)\big|_S\]
where $\gamma > 0$ is the localization strength.  Fix a step size $\eps > 0$ and generate, for each submanifold $S_i$ appearing in the decomposition of our observable, an SGLD chain $w_1^{(i)}, \ldots, w_T^{(i)}$ on $S_i$ with transition
\[w_{t+1}^{(i)} = w_t^{(i)} - \frac \eps 2 \nabla_{S_i}\big(n\beta K_n + \tfrac \gamma 2 \|w - w^*\|^2\big)(w_t^{(i)}) + \sqrt{\eps} \zeta_t, \qquad \zeta_t \sim N(0, \mr{Id}_{T_{w_t}S_i})\]
where $\nabla_{S_i}$ denotes the gradient along $S_i$.  Write $\widehat{\bb E}_{S_i,n}^{\mr{SGLD}}$ for the sample mean over these chains (after discarding burn-in), and similarly $\widehat{\bb E}_{W,n}^{\mr{SGLD}}$ for the full posterior.

\begin{definition}
    For a functional observable $\OO = \sum_i g_i \delta_{S_i}$, the \emph{SGLD susceptibility estimator} is
    \[\widehat\chi_{n,\beta}^{\mr{SGLD}}(\OO, \xi) = -\sum_i \left(\widehat{\bb E}_{S_i,n}^{\mr{SGLD}}\big(g_{i,n} \Delta K_n\big|_{S_i}\big) - \widehat{\bb E}_{S_i,n}^{\mr{SGLD}}\big(g_{i,n}\big|_{S_i}\big) \cdot \widehat{\bb E}_{W,n}^{\mr{SGLD}}(\Delta K_n)\right).\]
    This targets the renormalized susceptibility $\chi^{\mr{pop,ren}}_{n,\beta}(\OO,\xi)$; we do not attempt to estimate the partition function ratios in the present work.
\end{definition}

\begin{remark}
    The SGLD estimator $\widehat\chi_{n,\beta}^{\mr{SGLD}}$ introduces a fourth layer of approximation beyond those of the renormalized susceptibility estimator: the finite step size $\eps$, finite chain length $T$, and correlation between successive samples.  This layer is orthogonal to the population-to-empirical bridge and is not addressed in these notes.
\end{remark}

\begin{remark}
    Unlike the renormalized susceptibility estimator $\widehat\chi_{n,\beta}^{\mr{ren}}$, the SGLD estimator may depend on the choice of submanifold decomposition even in the large-$T$ limit.  Different decompositions $\OO = \sum_i g_i \delta_{S_i} = \sum_j g'_j \delta_{S'_j}$ require SGLD chains on different submanifolds, with different mixing rates, step size requirements, and discretization errors.  While both decompositions target the same renormalized population quantity $\chi^{\mr{pop,ren}}_{n,\beta}(\OO,\xi)$, the SGLD bias may differ, so the estimators need not agree at any finite $\eps$.
\end{remark}

\subsection{Convergence} \label{subsec:convergence_functional}

\subsubsection{Coupling kernels}

We will study the relationship between a population level susceptibility and its estimators for all perturbations at once using an intermediate step, in which we represent the susceptibility as the pairing of the perturbation against a kernel on the data manifold $X$.  Our approach generalizes the \emph{loss kernel} introduced in \cite{sing2}, although our kernels -- depending on a choice of observable -- will be asymmetric operators.

\begin{definition}[Coupling kernel] \label{def:coupling_kernel}
For a functional observable $\OO = g \delta_S$ with lift $\wt g$, the \emph{population coupling kernel} on $S$ is the hybrid kernel
\[\kappa^\OO_{n,\beta}(x,x') = \bb E_{\Pi^{\mr{pop},S}_{n,\beta}}[\wt g(x,w) f(x',w)] - \bb E_{\Pi^{\mr{pop},S}_{n,\beta}}[\wt g(x,w)] \cdot \bb E_{\sPiPop}[f(x',w)],\]
where the $\wt g$ factor is integrated under the restricted posterior $\Pi^{\mr{pop},S}_{n,\beta}$ and the $f$ factor in the cross term under the full posterior $\PiPop$.  The \emph{empirical coupling kernel} is defined analogously:
\[\widehat \kappa^\OO_{n,\beta}(x,x') = \bb E_{\Pi^{\mr{emp},S}_{n,\beta}}[\wt g(x,w) f(x',w)] - \bb E_{\Pi^{\mr{emp},S}_{n,\beta}}[\wt g(x,w)] \cdot \bb E_{\sPiEmp}[f(x',w)].\]
When $S = W$ both reduce to the ordinary posterior covariance $\mr{Cov}_{w \sim \Pi_{n,\beta}}(\wt g(x,w), f(x',w))$.
\end{definition}

\begin{definition}[Standard form] \label{def:sigma}
In standard form \cite[\S5.1]{WatanabeGreen}, suppose that we can write $f(x,u) = u^k a(x,u)$, $K(u) = u^{2k}$, $\varphi(u) = |u^h|b(u)$, $g(u) = u^{2j}$, and $\wt g(x,u) = u^j a_g(x,u)$, where $u$ is a local coordinate on a submanifold $S \sub W$.  For a multi-index $l \ge 0$ define
\[\lambda_l = \min_i \frac{l_i + h_i + 1}{2k_i}.\]
The RLCT is $\lambda = \lambda_0$.  We write $\sigma = \lambda_{j+k} - \lambda$ for the shift due to the insertion $w^{j+k}$.  When $j = k$ one has $\sigma = 1$.
\end{definition}

The coupling kernels depend on the parameters $n,\beta$ and we will study their relationship in the large $n\beta$ limit.  In \cite{sing2} an alternative limiting process is studied in terms of the sublevel sets of the loss of height $\eps$, in the $\eps \to 0$ limit.  These are closely related but not identical even in the limit.  We include the analogue of their definition here for reference.

\begin{definition}
The \emph{sharp-cutoff coupling kernel} on $S$ is
\[\kappa^\OO_\eps(x,x') = \eps^{-\sigma} \big[\bb E_{U(S_\eps)}[\wt g(x,w) f(x',w)] - \bb E_{U(S_\eps)}[\wt g(x,w)] \cdot \bb E_{U(W_\eps)}[f(x',w)]\big],\]
where $S_\eps = \{w \in S \colon K|_S(w) \le \eps\}$, $W_\eps = \{w \in W \colon K(w) \le \eps\}$, $U(\cdot)$ denotes the uniform distribution on its argument, and $\sigma$ is as in Definition \ref{def:sigma}.  When $S = W$ and $\sigma = 1$ this reduces to the sharp-cutoff loss kernel of \cite{sing2}.
\end{definition}

\subsubsection{From Kernels to Susceptibilities}

The coupling kernel controls the restricted covariance.  For a functional observable $\OO = g \cdot \delta_S$, expanding $g(w) = \int_X \wt g(x,w) q(x) \d x$ and $\Delta K(w) = \int_X \xi(x') f(x',w) q(x') \d x'$ in the definition of the restricted covariance (Definition \ref{def:restricted}) gives
\[\mr{Cov}^{\mr{res}}_{\sPiPop}(\OO, \Delta K) = \int_{X^2} \kappa^\OO_{n,\beta}(x,x') q(x) \xi(x') q(x') \d x \d x'.\]
The empirical counterpart $\Delta K_n$ differs from $\Delta K$ by an empirical process:
\[\Delta K_n(w) - \Delta K(w) = \frac{1}{\sqrt{n}} \psi_n^\xi(w),\]
where $\psi_n^\xi(w) = \frac{1}{\sqrt{n}}\sum_j [\xi(x_j)f(x_j,w) - \Delta K(w)]$.
Since $f$ is $L^4(q)$-valued real analytic and $\xi \in L^4(q)$, H\"older's inequality gives $\|\xi \cdot a_\alpha\|_2 \le \|\xi\|_4 \|a_\alpha\|_4$ for each Taylor coefficient $a_\alpha$ of $f$, so $\xi f$ is $L^2(q)$-valued real analytic in $w$.  Applying \cite[Theorem 5.8]{WatanabeGrey} gives $\sup_w |\Delta K_n(w) - \Delta K(w)| = O_p(1/\sqrt{n})$.

\begin{lemma} \label{kernel_to_susceptibility}
Suppose that $\widehat \kappa^\OO_{n,\beta}(x,x') \to \kappa^\OO_{n,\beta}(x,x')$ in probability for $q \otimes q$-almost all $(x,x') \in X^2$, and that there exists a function $M(x,x')$ with $\max(|\widehat \kappa^\OO_{n,\beta}(x,x')|, |\kappa^\OO_{n,\beta}(x,x')|) \le M(x,x')$ almost surely and
\[\int_{X^2} M(x,x') q(x) |\xi(x')| q(x') \d x \d x' < \infty.\]  Then
\[\mr{Cov}^{\mr{res}}_{\sPiEmp}(\OO_n, \Delta K) \to \mr{Cov}^{\mr{res}}_{\sPiPop}(\OO, \Delta K)\]
in probability as $n \to \infty$.
\end{lemma}

\begin{proof}
Decompose the difference at the covariance level:
\begin{align*}
\mr{Cov}^{\mr{res}}_{\sPiEmp}(\OO_n, \Delta K) - \mr{Cov}^{\mr{res}}_{\sPiPop}(\OO, \Delta K)
&= \Big[\mr{Cov}^{\mr{res}}_{\sPiEmp}(\OO_n, \Delta K) - \mr{Cov}^{\mr{res}}_{\sPiEmp}(\OO, \Delta K)\Big] \\
&\quad + \Big[\mr{Cov}^{\mr{res}}_{\sPiEmp}(\OO, \Delta K) - \mr{Cov}^{\mr{res}}_{\sPiPop}(\OO, \Delta K)\Big].
\end{align*}
The first summand isolates the effect of replacing the population coefficient $g$ by its empirical counterpart $g_n$, while the second isolates the effect of replacing the population posterior by the empirical posterior.

For the second summand, the coefficient $g(w) = \int_X \wt g(x,w) q(x) \d x$ is deterministic, so the restricted covariance factors through the coupling kernel integrated against $q(x) \xi(x') q(x') \d x \d x'$:
\[\mr{Cov}^{\mr{res}}_{\sPiEmp}(\OO, \Delta K) - \mr{Cov}^{\mr{res}}_{\sPiPop}(\OO, \Delta K) = \int_{X^2} (\widehat\kappa^\OO_{n,\beta}(x,x') - \kappa^\OO_{n,\beta}(x,x')) q(x) \xi(x') q(x') \d x \d x'.\]
This converges to zero in probability by the pointwise kernel convergence assumption and the dominated convergence theorem: the integrand is bounded by $2M(x,x') q(x) |\xi(x')| q(x')$, which is integrable, and any subsequence has a further subsequence along which $\widehat\kappa^\OO - \kappa^\OO \to 0$ almost surely (since convergence in probability implies almost sure convergence along a subsequence), so the ordinary dominated convergence theorem applies to that sub-subsequence.

For the first summand, the restricted covariance is linear in its first coefficient, so we must bound the restricted covariance $\mr{Cov}^{\mr{res}}_{\sPiEmp}(\OO_n - \OO, \Delta K)$.  Since $\wt g(x,w)$ is $L^2(q)$-valued real analytic in $w$ we may apply empirical process bound \cite[Theorem 5.8]{WatanabeGrey}.  This bound tells us that there exists a constant $C$ so that
\[\bb E(\sup_{w \in S} n |g_n(w) - g(w)|^2) \le C.\]
By Markov's inequality we conclude that  $\sup_{w \in S} |g_n(w) - g(w)| = O_p(1/\sqrt{n})$.  The restricted covariance with coefficient $g_n - g$ is therefore bounded by an expression of order $O_p(1/\sqrt{n})$ times the posterior expectation of $|\Delta K|$, which is $O(1)$.  Hence this summand also converges to zero in probability as $n \to \infty$.
\end{proof}

\begin{remark} \label{rmk:domination}
    The domination condition is automatic given our assumptions.  Cauchy--Schwarz gives
    \[|\kappa^{\OO}_{n,\beta}(x,x')| \le \sup_{w \in S} |\wt g(x,w)| \cdot \sup_{w \in W} |f(x',w)|\]
    and likewise for $\widehat \kappa^\OO_{n,\beta}$.  Taking $M(x,x') = \sup_{w \in S} |\wt g(x,w)| \cdot \sup_{w \in W} |f(x',w)|$, the integral $\int_{X^2} M\, q(x) |\xi(x')| q(x') \d x \d x'$ factorizes by Fubini as
    \[\left(\int_X \sup_{w \in S} |\wt g(x,w)| q(x) \d x\right) \cdot \left(\int_X \sup_{w \in W} |f(x',w)| |\xi(x')| q(x') \d x'\right).\]
    The first factor is finite since $\sup_{w \in S} |\wt g(\cdot, w)| \in L^4(q) \sub L^1(q)$ by Proposition \ref{Ls_pointwise}.  For the second factor, $\sup_{w \in W} |f(\cdot,w)| \in L^4(q)$ by Proposition \ref{Ls_pointwise}, so since $\xi \in L^4(q)$, H\"older's inequality gives $\int \sup_w |f(x',w)| |\xi(x')| q(x') \d x' \le \|\sup_w |f|\|_{L^4(q)} \|\xi\|_{L^4(q)} < \infty$.
\end{remark}

\begin{remark} \label{rmk:holder_tradeoff}
    The arguments above require that $\xi f$ is $L^2(q)$-valued real analytic.  By H\"older's inequality this holds whenever $\xi \in L^a(q)$ and $f$ is $L^b(q)$-valued real analytic with $\frac{1}{a} + \frac{1}{b} = \frac{1}{2}$.  Our standing assumption (Assumption \ref{standing_assumption}) takes $b = 4$, requiring $a = 4$, i.e.\ $\xi \in L^4_0(q)$.  An alternative is to assume $f$ is only $L^2(q)$-valued real analytic and require $\xi \in L^\infty(q)$; this is the other extreme of the same H\"older tradeoff.  Any pair $(a,b)$ with $a,b > 2$ and $\frac{1}{a} + \frac{1}{b} = \frac{1}{2}$ suffices for consistency.  For asymptotic unbiasedness (Theorem \ref{asymptotic_unbiasedness}) one needs $\xi f$ to be $L^4(q)$-valued real analytic, which requires the stronger condition $\frac{1}{a} + \frac{1}{b} = \frac{1}{4}$.
\end{remark}

\subsubsection{Moment lemmas} \label{moment_lemma_section}

In this section we will prove two lemmas by which we may explicitly analyze the asymptotic behavior of the moments appearing in first the population, then the empirical coupling kernels.  Our argument in this section follows an approach communicated to us by Zach Furman, to appear in a forthcoming update to \cite[Appendix A]{sing2}.  In the next subsection we will use these lemmas to directly compare the population and empirical coupling kernels, and therefore the population level susceptibility and its empirical estimator for a functional observable.

\begin{lemma} \label{population_moment_scaling}
Using the notation of Definition \ref{def:sigma}, let $\psi(x,w) = w^l c(x,w)$ where $l \ge 0$ is a multi-index and $c$ is real analytic with $c(x,0) \ne 0$ for $q$-a.e.\ $x$.  Write $\tau = \lambda_l - \lambda$ and let $m_l$ be the multiplicity of $\lambda_l$.  Then
\[\bb E_{w \sim \sPiPop}[\psi(x,w)] = C_l (n\beta)^{-\tau} (\log(n\beta))^{m_l - m} \bb E_{w \sim D_l}[c(x,w)] + o\big((n\beta)^{-\tau} (\log(n\beta))^{m_l - m}\big),\]
where $D$ is Watanabe's asymptotic distribution \cite[Remark 33]{WatanabeGreen}, $D_l$ is the asymptotic distribution for the weight $w^{l+h}$, $\bb E_{w \sim D_l}$ denotes the normalized expectation $\int c D_l \d w / \int D_l \d w$, and
\[C_l = \frac{\Gamma(\lambda_l)}{\Gamma(\lambda)}\cdot\frac{\int D_l(w) \d w}{\int D(w) \d w}\]
is independent of $n,\beta$.  When $m_l = m$ the log factor is absent.
\end{lemma}

\begin{proof}
The expectation is
\[\bb E_{\sPiPop}[\psi] = \frac{\int w^{l+h} c(x,w) b(w) e^{-n\beta w^{2k}} \d w}{\int |w^h| b(w) e^{-n\beta w^{2k}} \d w}.\]
We identify the numerator with
\[\int_0^\infty \frac{\d t}{n\beta}\int \delta\Big(\frac{t}{n\beta} - w^{2k}\Big) w^{l+h} c(x,w) b(w) e^{-t} \d w.\]
By \cite[Theorem 4.7]{WatanabeGrey} the inner integral is $(t/(n\beta))^{\lambda_l - 1}(\log(n\beta/t))^{m_l-1}(\int D_l(w) c(x,w) \d w + o(1))$, where $D_l$ is the asymptotic distribution for the weight $w^{l+h}$.  Substituting and using $(\log(n\beta/t))^{m_l-1} = (\log(n\beta))^{m_l-1} + O((\log(n\beta))^{m_l-2})$, the numerator is
\[\frac{(\log(n\beta))^{m_l-1}}{(n\beta)^{\lambda_l}}\Gamma(\lambda_l)\int D_l(w) c(x,w) \d w + o\Big(\frac{(\log(n\beta))^{m_l-1}}{(n\beta)^{\lambda_l}}\Big).\]
The same argument applied to the denominator (with $l = 0$ and $c = 1$) gives
\[\frac{(\log(n\beta))^{m-1}}{(n\beta)^{\lambda}}\Gamma(\lambda)\int D \d w + o\Big(\frac{(\log(n\beta))^{m-1}}{(n\beta)^{\lambda}}\Big).\]
Dividing these two expressions yields the claim.
\end{proof}

The following lemma is the empirical counterpart of Lemma \ref{population_moment_scaling}, generalizing Furman's argument to insertions $w^l$ that need not be proportional to $k$.

\begin{lemma} \label{empirical_moment_scaling}
Using the notation of Definition \ref{def:sigma} and Lemma \ref{population_moment_scaling}, let $\psi(x,w) = w^l c(x,w)$ and write $\tau = \lambda_l - \lambda$.  Define the \emph{generalized renormalized posterior}
\[D^{l}_{\zeta_n}(w) = \frac{D_l(w) S_{\lambda_l}(\zeta_n(w))}{\int D(w) S_\lambda(\zeta_n(w)) \d w},\]
where $S_\alpha(a) = \int_0^\infty t^{\alpha - 1}\exp(-t + a\sqrt{\beta t}) \d t$ is the fluctuation function.  Then under the conditions of \cite[Theorem 10]{WatanabeGreen},
\[\bb E_{w \sim \sPiEmp}[\psi(x,w)] = (n\beta)^{-\tau} (\log(n\beta))^{m_l - m} \bb E_{w \sim D^{l}_{\zeta_n}}[c(x,w)] + o_p\big((n\beta)^{-\tau} (\log(n\beta))^{m_l - m}\big).\]
\end{lemma}

\begin{proof}
Write
$\Omega(w) = |w^h|b(w)\exp(-n\beta K_n(w))$
for the normalized posterior function.  The numerator of $\bb E_{\sPiEmp}[\psi]$ is
\[N = \int w^{l+h} c(x,w) b(w) \exp(-n\beta w^{2k} + \sqrt{n}\,\beta\, w^k \zeta_n(w)) \d w.\]
Introducing $t = n\beta w^{2k}$ as an auxiliary variable and applying \cite[Theorem 10]{WatanabeGreen} to the level-set integral with weight $w^{l+h}$ gives
\[N = \frac{(\log(n\beta))^{m_l-1}}{(n\beta)^{\lambda_l}} \int_0^\infty \d t \int D_l(w) c(x,w) t^{\lambda_l - 1}\exp(-t + \sqrt{\beta t} \zeta_n(w)) \d w + o_p\Big(\frac{(\log(n\beta))^{m_l-1}}{(n\beta)^{\lambda_l}}\Big).\]
The denominator $N_0 = \int \Omega(w) \d w$ gives the same expression with $l = 0$, $c = 1$, $\lambda_l$ replaced by $\lambda$, and $D_l$ replaced by $D$.  Dividing the two expressions yields the claim.
\end{proof}

\begin{lemma} \label{population_and_renormalized_posteriors}
For any bounded measurable function $c$ on $W$, as $\beta \to 0$,
\[\bb E_{D^l_{\zeta_n}}[c] = C_l \bb E_{D_l}[c] + O_p(\sqrt\beta),\]
where $\bb E_{D^l_{\zeta_n}}[c] = \int D^l_{\zeta_n}(w) c(w) \d w$ is the unnormalized integral.
\end{lemma}

\begin{proof}
Expand $S_\alpha(a) = \int_0^\infty t^{\alpha-1}\exp(-t + a\sqrt{\beta t}) \d t = \Gamma(\alpha)(1 + a\sqrt{\beta} k_\alpha + O(\beta))$ where $k_\alpha = \Gamma(\alpha + 1/2)/\Gamma(\alpha)$.  Since $\zeta_n$ is tight in the sup norm on the compact set $W$, the expansion holds uniformly in $w$ when $a = \zeta_n(w)$.  By definition,
\[\bb E_{D^l_{\zeta_n}}[c] = \frac{\int D_l(w) S_{\lambda_l}(\zeta_n(w)) c(w) \d w}{\int D(w) S_\lambda(\zeta_n(w)) \d w}.\]
Substituting the expansion of $S_\alpha$ into numerator and denominator separately:
\begin{align*}
\int D_l(w) S_{\lambda_l}(\zeta_n(w)) c(w) \d w &= \Gamma(\lambda_l)\Big[\int D_l c \d w + \sqrt\beta k_{\lambda_l} \int D_l \zeta_n c \d w + O_p(\beta)\Big], \\
\int D(w) S_\lambda(\zeta_n(w)) \d w &= \Gamma(\lambda)\Big[\int D \d w + \sqrt\beta k_\lambda \int D \zeta_n \d w + O_p(\beta)\Big].
\end{align*}
The integrals $\int D_l \zeta_n c \d w$ and $\int D \zeta_n \d w$ are bounded in probability (since $\zeta_n$ is bounded on $W$ and $D_l$, $D$ are integrable).  Dividing and expanding:
\[\bb E_{D^l_{\zeta_n}}[c] = \frac{\Gamma(\lambda_l)}{\Gamma(\lambda)} \cdot \frac{\int D_l c \d w}{\int D \d w} + O_p(\sqrt\beta) = C_l \bb E_{D_l}[c] + O_p(\sqrt\beta),\]
where $C_l = \frac{\Gamma(\lambda_l)}{\Gamma(\lambda)}\cdot\frac{\int D_l}{\int D}$.
\end{proof}

\begin{remark}
The assumption $\beta \to 0$ on the inverse temperature was necessary for the argument above.  Note that such conditions already appear in the work of Watanabe, in particular in \cite{WatanabeWBIC} where $\beta$ is taken to be proportional to $1/\log(n)$.
\end{remark}

\subsubsection{Coupling kernel and susceptibility convergence}
We now have all the ingredients that we need to directly address the relationship between population and empirical coupling kernels, and therefore by Lemma \ref{kernel_to_susceptibility} the relationship between the empirical susceptibility and its estimator for a functional observable.

We begin by introducing a necessary condition on the lifts $\wt g$, which we will require for the generalization of Watanabe's results on empirical processes.

\begin{definition} \label{def:rfv}
An $L^2(q)$-valued real analytic function $\psi$ on $W$ has \emph{relatively finite variance} (RFV) if there exists $c_0 > 0$ such that for all $w \in W$ away from the global minimum locus of $K$
\[\bb E_q[\psi(X,w)^2] \le c_0 |\bb E_q[\psi(X,w)]|.\]
\end{definition}

\begin{theorem} \label{coupling_kernel_convergence}
Let $\OO = g \delta_S$ be a functional observable.  Suppose that $\wt g$ is $L^2(q)$-valued real analytic in $w$ and has RFV.  Then as $n\beta \to \infty$ and $\beta \to 0$,
\[\widehat \kappa^\OO_{n,\beta}(x,x') - \kappa^\OO_{n,\beta}(x,x') \to 0\]
in probability for $q$-almost all $(x,x') \in X^2$.
\end{theorem}

\begin{proof}
The hybrid kernel $\kappa^\OO_{n,\beta}$ is a difference of two terms: the mixed moment $\bb E_{\Pi^{\mr{pop},S}_{n,\beta}}[\wt g f]$ on $S$ and the product of first moments $\bb E_{\Pi^{\mr{pop},S}_{n,\beta}}[\wt g] \cdot \bb E_{\sPiPop}[f]$ on $S$ and $W$ respectively.  The same decomposition holds for $\widehat\kappa^\OO$ with empirical posteriors, and we show convergence to zero for the two terms individually.

We first consider the mixed moment.  By Proposition \ref{Ls_pointwise} the function $\wt g(x,w) f(x',w)$ is real analytic in $w \in S$ for $q\otimes q$-almost every $(x,x') \in X^2$.  After resolving $K|_S$ in standard form (Definition \ref{def:sigma}) this insertion has monomial form $u^{j+k} a_g(x,u) a(x',u)$, whose leading coefficient $a_g(x,0)\,a(x',0)$ is nonzero for $q\otimes q$-a.e.\ $(x, x')$ by independence of the null sets.  The moment lemmas (Lemmas \ref{population_moment_scaling} and \ref{empirical_moment_scaling}) therefore apply to this insertion on $S$.  By Lemma \ref{population_and_renormalized_posteriors} (applied on $S$), the leading coefficients of the empirical and population moments agree as $\beta \to 0$, giving $\bb E_{\Pi^{\mr{emp},S}_{n,\beta}}[\wt g f] - \bb E_{\Pi^{\mr{pop},S}_{n,\beta}}[\wt g f] \to 0$ in probability.

For the product of first moments: the same moment comparison gives $\bb E_{\Pi^{\mr{emp},S}_{n,\beta}}[\wt g] \to \bb E_{\Pi^{\mr{pop},S}_{n,\beta}}[\wt g]$ in probability (a moment on $S$) and $\bb E_{\sPiEmp}[f] \to \bb E_{\sPiPop}[f]$ in probability (a moment on $W$).  Writing
\begin{align*}
&\bb E_{\Pi^{\mr{emp},S}_{n,\beta}}[\wt g] \cdot \bb E_{\sPiEmp}[f] - \bb E_{\Pi^{\mr{pop},S}_{n,\beta}}[\wt g] \cdot \bb E_{\sPiPop}[f] \\
&\quad = \big(\bb E_{\Pi^{\mr{emp},S}_{n,\beta}}[\wt g] - \bb E_{\Pi^{\mr{pop},S}_{n,\beta}}[\wt g]\big) \cdot \bb E_{\sPiEmp}[f] \\
&\qquad + \bb E_{\Pi^{\mr{pop},S}_{n,\beta}}[\wt g] \cdot \big(\bb E_{\sPiEmp}[f] - \bb E_{\sPiPop}[f]\big),
\end{align*}
and noting that $\bb E_{\sPiEmp}[f]$ is bounded in probability (it is a posterior expectation of $f$, which is bounded on compact $W$), the product converges to zero in probability.
\end{proof}

\begin{theorem} \label{convergence_zeroth_order}
Let $\OO = \sum_i g_i \delta_{S_i}$ be a functional observable.  Suppose that each $\wt g_i$ is real analytic in $w$ and has RFV.  Then as $n\beta \to \infty$ and $\beta \to 0$,
\[\mr{Cov}^{\mr{res}}_{\sPiEmp}(\OO_n, \Delta K_n) - \mr{Cov}^{\mr{res}}_{\sPiPop}(\OO, \Delta K) \to 0\]
in probability.  In other words, the renormalized susceptibility estimator is consistent.
\end{theorem}

\begin{proof}
By linearity it suffices to consider $\OO = g \delta_S$.  By the argument of Lemma \ref{kernel_to_susceptibility} using Watanabe's empirical process bound, $\sup_w |\Delta K_n(w) - \Delta K(w)| = O_p(1/\sqrt{n})$, so it suffices to prove convergence with $\Delta K$ in place of $\Delta K_n$ on the empirical side.  The coupling kernel convergence (Theorem \ref{coupling_kernel_convergence}) and the domination condition (Remark \ref{rmk:domination}) verify the hypotheses of Lemma \ref{kernel_to_susceptibility}, which gives the result.
\end{proof}

\begin{remark} \label{exact_convergence}
When $\OO = \sum_i g_i \delta_{S_i}$ is deterministic (each $g_i$ independent of $q$), a more direct argument is available that avoids the coupling kernel.  The key ideas are as follows.

Fix a log resolution $\varpi$ for $K|_S$ with local coordinate $u$ in which 
\[K|_S(\varpi(u)) = u^{2k} \text{ and } \varphi(\varpi(u))|\varpi'(u)| = \phi(u) u^e,\]
with $\phi > 0$ smooth.  By \cite[Main Theorem 1]{WatanabeGrey}, $nK_n(\varpi(u)) = nu^{2k} - \sqrt{n} u^k \zeta_n(u)$ where $\zeta_n \to \zeta$ in law with respect to the sup norm.\footnote{We write $\zeta_n$ for Watanabe's empirical process (denoted $\xi_n$ in \cite{WatanabeGrey,WatanabeGreen,sing2}) to avoid conflict with the perturbation $\xi \in T_q\Delta(X)$ from \S2.}

For a smooth deterministic insertion $h$ on $S$, the partition function integrals become $Z^p(n,\zeta_n,\psi)$ with $\psi = h \circ \varpi$.  By \cite[Theorem 4.9]{WatanabeGrey}, each is approximated by its central part, which factors as $\gamma_b (\log n)^{r-1} n^{-(\lambda+p)} \cdot \Phi^p(\zeta_n,\psi)$.  Since $\zeta_n \to \zeta$ in law, the continuous mapping theorem gives $\Phi^p(\zeta_n,\psi) \to \Phi^p(\zeta,\psi)$ in law.  The deterministic prefactor cancels in the ratio.  As $\beta \to 0$, $\Phi^p(\zeta,\psi)/\Phi^0(\zeta,1) \to \Phi^p(0,\psi)/\Phi^0(0,1)$ because $\zeta$ enters only through $e^{\sqrt{\beta t}\zeta_0(y)} \to 1$ uniformly.

For the $\Delta K_n$ insertion, write $\Delta K_n = \Delta K + (\Delta K_n - \Delta K)$.  The $\Delta K$ part is a fixed smooth function, handled as above.  The remainder $\Delta K_n - \Delta K$ satisfies $\sup_w |\Delta K_n(w) - \Delta K(w)| = O_p(1/\sqrt{n})$ by \cite[Theorem 5.8]{WatanabeGrey}, so its contribution is subleading.
\end{remark}

\begin{theorem} \label{asymptotic_unbiasedness}
Suppose that $\xi \in L^a(q)$ and that $f$ and each $\wt g_i$ are $L^b(q)$-valued real analytic on $W$, where $\frac{1}{a} + \frac{1}{b} = \frac{1}{4}$.  Then under the hypotheses of Theorem \ref{convergence_zeroth_order},
\[\bb E\Big[\mr{Cov}^{\mr{res}}_{\sPiEmp}(\OO_n, \Delta K_n)\Big] - \mr{Cov}^{\mr{res}}_{\sPiPop}(\OO, \Delta K) \to 0\]
as $n\beta \to \infty$ and $\beta \to 0$.  In other words, the renormalized susceptibility estimator is asymptotically unbiased.
\end{theorem}

\begin{proof}
By Theorem \ref{convergence_zeroth_order}, $X_n := \mr{Cov}^{\mr{res}}_{\sPiEmp}(\OO_n, \Delta K_n) - \mr{Cov}^{\mr{res}}_{\sPiPop}(\OO, \Delta K) \to 0$ in probability.  It suffices to show that $\{X_n\}$ is uniformly integrable, since convergence in probability plus uniform integrability implies convergence of expectations.

The restricted covariance is bounded by
\[|X_n| \le |\mr{Cov}^{\mr{res}}_{\sPiEmp}(\OO_n, \Delta K_n)| + |\mr{Cov}^{\mr{res}}_{\sPiPop}(\OO, \Delta K)| \le 2 \sup_W |g_n| \sup_W |\Delta K_n| + C_0\]
where $C_0 = |\mr{Cov}^{\mr{res}}_{\sPiPop}(\OO, \Delta K)|$ is deterministic.  The H\"older condition $\frac{1}{a} + \frac{1}{b} = \frac{1}{4}$ ensures that $\xi f$ is $L^4(q)$-valued real analytic, and hence by \cite[Theorem 5.8]{WatanabeGrey} applied with $s = 4$, there exist constants $C$ and $C'$ such that $\bb E[\sup_W |g_n|^4] \le C$ and $\bb E[\sup_W |\Delta K_n|^4] \le C'$ uniformly in $n$.  By Cauchy--Schwarz,
\[\bb E[(\sup_W |g_n| \sup_W |\Delta K_n|)^2] \le \bb E[\sup_W |g_n|^4]^{1/2} \bb E[\sup_W |\Delta K_n|^4]^{1/2} \le (CC')^{1/2}.\]
Hence $\bb E[|X_n|^2]$ is uniformly bounded in $n$, which implies uniform integrability of $\{X_n\}$.
\end{proof}

\section{Differential Observables} \label{sec:differential}

We now treat differential observables $\OO = \sum_i P_i \delta_{S_i}$ with $M_i = \mr{ord}(P_i) > 0$ for some $i$.  Our strategy is to reduce to the functional case in the following way.  We argue that we may reduce our susceptibility estimation problem to a local problem in parameter space.  In local charts we will use integration by parts to reduce to observables for which all derivatives act in directions normal to the support of the observable.  We then identify such observables with sums of functional observables by applying the Leibniz rule upon evaluation of the restricted covariance. 

Following this plan, we will first show that every differential observable admits a local normal form.  That is, near each point of $S$, it can be expressed as a finite sum of normal-derivative terms applied to functional observables.

\begin{lemma} \label{tangential_normal_decomposition}
Let $\OO = P \delta_S$ where $P$ is a linear differential operator of order $M$ with $L^2(q)$-valued real analytic coefficients, and suppose that all coefficient lifts and their derivatives up to order $M$ have RFV.  Then for each $s \in S$ there exists an open neighborhood $U_s$ of $s$ in $W$ such that $\OO|_{U_s}$ is a finite sum
\[\OO|_{U_s} = \sum_j \dd^{\beta_j} (\eta_j \delta_{T_j})\]
where each $T_j$ is $S \cap U_s$ or a boundary stratum thereof, each $\dd^{\beta_j}$ is a differential operator purely normal to $S$ in the chart coordinates on $U_s$, and each $\eta_j \delta_{T_j}$ is a functional observable whose lift has RFV.
\end{lemma}

\begin{proof}
Let $m = \dim S$ and $r = d - m$.  Since $S$ is neatly embedded in $W$, for each $s \in S$ there exists a chart $(U_s, \phi_s)$ in which $\phi_s(U_s \cap S) \sub \RR^m \times \{0\}$ and $\phi_s(U_s \cap W) \sub \RR_{\ge 0}^p \times \RR^{d-p}$ for some $p \le d$.\footnote{In the terminology of Melrose \cite{Melrose}, $S$ is a \emph{$p$-submanifold} of the manifold with corners $W$.}  Working in this chart, write $v = (v_1, \ldots, v_{p+q})$ for the tangential coordinates on $S$ and $u = (u_1, \ldots, u_r)$ for the normal coordinates, so that $S \cap U_s = \{u = 0\}$ and $W \cap U_s = \RR^p_{\ge 0} \times \RR^q \times \RR^r$.  In these coordinates $P$ becomes a smooth-coefficient operator; expanding it produces finitely many monomial terms $g_\gamma \dd^\alpha \dd^\beta$ with $\dd^\alpha$ tangential and $\dd^\beta$ normal.  It suffices to treat each monomial.  We induct on $|\alpha|$, the tangential order.

The base case $|\alpha| = 0$ is immediate: $g_\gamma \dd^\beta \delta_S = \dd^\beta(g_\gamma \delta_S)$ with $g_\gamma \delta_S$ functional.

For the induction step, write $\alpha = \alpha' + e_j$ for some tangential coordinate direction $j$.  The boundary faces of $S \cap U_s$ are the faces $F_\ell = \{v_\ell = 0\} \cap S \cap U_s$ for $\ell = 1, \ldots, p$, each with outward unit normal $\nu^{(\ell)} = -e_\ell$.  Higher-codimension strata have measure zero in $\dd S$.  For any test function $\Phi$ on $S \cap U_s$, integration by parts gives
\[\int_{S \cap U_s} g_\gamma(v) \dd_{v_j} \Phi(v) \d v = -\int_{S \cap U_s} (\dd_{v_j} g_\gamma(v)) \Phi(v) \d v + \sum_{\ell=1}^p \int_{F_\ell} g_\gamma(v) \Phi(v) \nu^{(\ell)}_j \d\sigma_\ell,\]
where $\d v$ denotes the volume form on $S \cap U_s$ in the chart coordinates.  The bulk term has tangential order $|\alpha| - 1$ with coefficient $\dd_{v_j} g_\gamma$, whose lift has RFV by hypothesis.  Each boundary term has tangential order $|\alpha| - 1$ on $F_\ell$ with coefficient $g_\gamma|_{F_\ell} \nu^{(\ell)}_j$, which inherits RFV.  Since $\dd^\beta$ remains purely normal to $F_\ell$, the inductive hypothesis applies.

After at most $|\alpha|$ steps, every term has the form $\dd^\beta(\eta \delta_T)$ with $T$ a stratum of $S \cap U_s$, $\dd^\beta$ a chart-normal operator, and $\eta \delta_T$ functional with RFV lift.
\end{proof}

By Lemma \ref{tangential_normal_decomposition}, in a neighborhood of each point of $S$ the observable $\OO$ is a finite sum of terms $\dd^\beta(\eta \delta_T)$ with $\dd^\beta$ chart-normal.  We now define the Leibniz coefficients, the local renormalized susceptibility, and then the global quantities.

\begin{definition} \label{def:leibniz_coefficients}
Let $\dd^\beta$ be a differential operator purely normal to $S$ in chart coordinates, and let $F = e^{-n\beta G}\varphi$ with $G = K$ or $G = K_n$.  Assume $\varphi|_S > 0$.  Writing $h = -n\beta G + \log\varphi$ so that $F = e^h$, the Leibniz rule gives $\dd^\beta F|_{u=0}$ as a finite sum
\begin{equation} \label{Leibniz_equation}
    e^{-n\beta G|_S}\,\varphi|_S \sum_j (-n\beta)^{r_j} Q_j^G(v),
\end{equation}
where each $Q_j^G$ is a polynomial in normal derivatives $(\dd^{\gamma_i} G)|_{u=0}$ and $(\dd^\delta \log\varphi)|_{u=0}$, and $r_j$ counts the number of $G$-derivative factors.  We call the $Q_j^G$ the \emph{Leibniz coefficients} of $\dd^\beta$ with respect to $G$.
\end{definition}

\begin{definition} \label{def:restricted_differential}
Let $\OO' = g \delta_S$ be a functional observable, $\dd^\beta$ a chart-normal operator of order $|\beta|$, and $M \ge |\beta|$ an integer (the global order of the ambient differential observable).  The \emph{local restricted covariance} of $\dd^\beta(\OO')$ with $\Delta K$ at order $M$ is
\[\mr{Cov}^{\mr{res}}_{\Pi_{n,\beta}}(\dd^\beta(\OO'), \Delta K; M) = \frac{1}{(n\beta)^M}\sum_j (-1)^{|\beta|} (-n\beta)^{r_j} \mr{Cov}^{\mr{res}}_{\Pi_{n,\beta}}(g Q_j^G \delta_S, \Delta K),\]
where $\Pi_{n,\beta}$ stands for $\PiPop$ or $\PiEmp$, the $Q_j^G$ are the Leibniz coefficients of $\dd^\beta$, with $G = K$ on the population side and $G = K_n$ on the empirical side.  Since $r_j \le |\beta| \le M$ the prefactors $(n\beta)^{r_j - M} \le 1$ are bounded.  The local \emph{renormalized susceptibility} and \emph{renormalized susceptibility estimator} are $-\mr{Cov}^{\mr{res}}_{\sPiPop}(\dd^\beta(\OO'), \Delta K; M)$ and $-\mr{Cov}^{\mr{res}}_{\sPiEmp}(\dd^\beta(\OO')_n, \Delta K_n; M)$ respectively.
\end{definition}

\begin{definition} \label{def:global_renormalized}
For a general differential observable $\OO = P \delta_S$ of pure order $M$, choose a finite set of open charts $\{U_i\}$ in $W$ covering $S$ as in Lemma \ref{tangential_normal_decomposition} and a partition of unity $\{\rho_i\}$ subordinate to $\{U_i\}$.  Since $\langle P\delta_S, F \rangle = \sum_i \langle \rho_i P\delta_S, F \rangle$, Lemma \ref{tangential_normal_decomposition} applied to $(P)|_{U_i}$ produces a finite sum $\sum_j \dd^{\beta_{ij}}(\eta_{ij} \delta_{T_{ij}})$ with $\dd^{\beta_{ij}}$ chart-normal and $\eta_{ij}$ incorporating $\rho_i g_\alpha$ and its tangential derivatives.  The \emph{renormalized susceptibility} is
\[\chi^{\mr{pop,ren}}_{n,\beta}(\OO, \xi) = -\sum_{i,j} \mr{Cov}^{\mr{res}}_{\sPiPop}(\rho_i \dd^{\beta_{ij}}(\eta_{ij} \delta_{T_{ij}}), \Delta K; M)\]
and the \emph{renormalized estimator} is
\[\widehat\chi_{n,\beta}^{\mr{ren}}(\OO, \xi) = -\sum_{i,j} \mr{Cov}^{\mr{res}}_{\sPiEmp}(\rho_i\dd^{\beta_{ij}}((\eta_{ij})_n \delta_{T_{ij}}), \Delta K_n; M)\]
using Definition \ref{def:restricted_differential} with the global order $M$ for each summand.
\end{definition}

\begin{remark}
The definition depends on the choice of cover and partition of unity.  The resulting quantity is independent of these choices, since the distributional pairing $\langle P\delta_S, F \rangle = \sum_i \langle \rho_i P\delta_S, F \rangle$ is intrinsic.
\end{remark}

\begin{remark} \label{rmk:analyticity_partition}
The Leibniz coefficients $Q_j^G$ and the functional coefficients $\eta_{ij}$ are real analytic in the chart coordinates: the $Q_j^G$ are polynomials in normal derivatives of $K$ and of $\log\varphi$, while the $\eta_{ij}$ arise from the original coefficients of $P$ and their tangential derivatives.  The partition of unity factors $\rho_i$ are smooth but \emph{not} real analytic, so we cannot apply Theorem \ref{convergence_zeroth_order} directly.  However the partition of unity is deterministic, so it appears identically in the population and empirical quantities.  This will allow us to check our convergence theorem using purely the local quantities defined in terms of the real analytic Leibniz coefficients.
\end{remark}

\begin{lemma} \label{smooth_prefactor}
Let $\OO = g \delta_T$ be a functional observable satisfying the hypotheses of Theorem \ref{convergence_zeroth_order}, and let $\rho$ be a smooth function on $W$.  Then
\[\mr{Cov}^{\mr{res}}_{\sPiEmp}((\rho \OO)_n, \Delta K_n) - \mr{Cov}^{\mr{res}}_{\sPiPop}(\rho \OO, \Delta K) \to 0\]
in probability as $n\beta \to \infty$ and $\beta \to 0$.
\end{lemma}

\begin{proof}
The lift of $\rho g$ is $\rho(w) \wt g(x,w)$.  Since $\rho$ is $x$-independent, the empirical coefficient is $(\rho g)_n = \rho g_n$.  We follow the proof of Theorem \ref{convergence_zeroth_order} (via Lemma \ref{kernel_to_susceptibility}) for the coefficient $\rho g$.  Recall that we split the expression we wish to show converges to zero into a sum of two terms, one involving the difference of empirical and population coefficients and another involving the difference of empirical and population kernels.

For the coefficient convergence term: $(\rho g)_n - \rho g = \rho(g_n - g)$, so $\sup_w |(\rho g)_n - \rho g| \le \sup_w |\rho| \cdot \sup_w |g_n - g|$.  Since $\rho$ is smooth on the compact set $W$, $\sup_w |\rho| < \infty$, and the empirical process bound for $\wt g$ gives $\sup_w |g_n - g| = O_p(1/\sqrt{n})$ as before.

For the kernel convergence term: the coupling kernel $\kappa^{\rho\OO}$ for coefficient $\rho g$ involves the posterior expectations $\bb E_{\Pi^{\mr{pop},S}_{n,\beta}}[\rho \wt g f]$ and $\bb E_{\Pi^{\mr{pop},S}_{n,\beta}}[\rho \wt g] \cdot \bb E_{\sPiPop}[f]$.  Since $\rho$ is $x$-independent, it enters the partition function integrands as a bounded smooth multiplicative factor.  The empirical-to-population moment comparison is unaffected, since $\rho$ is the same on both sides.
\end{proof}

\begin{theorem} \label{convergence_differential}
Let $\OO = P \delta_S$ where $P$ is a linear differential operator of order $M$.  Suppose that all coefficient lifts of $P$, and their derivatives up to order $M$, have RFV.  Then as $n\beta \to \infty$ and $\beta \to 0$,
\[\mr{Cov}^{\mr{res}}_{\sPiEmp}(\OO_n, \Delta K_n) - \mr{Cov}^{\mr{res}}_{\sPiPop}(\OO, \Delta K) \to 0\]
in probability.
\end{theorem}

\begin{proof}
By Definition \ref{def:global_renormalized}, the renormalized susceptibility is a finite sum over $i,j$ of terms involving $\rho_i \mr{Cov}^{\mr{res}}_\Pi(\dd^{\beta_{ij}}(\eta_{ij} \delta_{T_{ij}}), \Delta K; M)$.  By Remark \ref{rmk:analyticity_partition}, each $\eta_{ij}$ has an analytic lift with RFV, and each $\dd^{\beta_{ij}}(\eta_{ij}\delta_{T_{ij}})$ is a local term to which Definition \ref{def:restricted_differential} applies with the global order $M$.  By Lemma \ref{smooth_prefactor}, multiplication by $\rho_i$ preserves convergence, so it suffices to prove convergence for each local term $\dd^{\beta_{ij}}(\eta_{ij}\delta_{T_{ij}})$ with $\dd^{\beta_{ij}}$ purely normal of order $|\beta_{ij}| \le M$.

By Definition \ref{def:restricted_differential}, the difference of empirical and population local restricted covariances of $\dd^\beta(\eta\delta_T)$ is
\[\sum_j (-1)^{|\beta|}(n\beta)^{r_j - M} \Big[\mr{Cov}^{\mr{res}}_{\sPiEmp}(\eta_n Q_j^{K_n} \delta_T, \Delta K_n) - \mr{Cov}^{\mr{res}}_{\sPiPop}(\eta Q_j^K \delta_T, \Delta K)\Big].\]

For the leading terms ($r_j = |\beta| = M$) we decompose the bracketed difference as
\begin{align*}
&\mr{Cov}^{\mr{res}}_{\sPiEmp}(\eta_n Q_j^{K_n} \delta_T, \Delta K_n) - \mr{Cov}^{\mr{res}}_{\sPiPop}(\eta Q_j^K \delta_T, \Delta K) \\
&\quad = \mr{Cov}^{\mr{res}}_{\sPiEmp}\big(\eta_n (Q_j^{K_n} - Q_j^K) \delta_T, \Delta K_n\big) \\
&\qquad + \mr{Cov}^{\mr{res}}_{\sPiEmp}(\eta_n Q_j^K \delta_T, \Delta K_n) - \mr{Cov}^{\mr{res}}_{\sPiPop}(\eta Q_j^K \delta_T, \Delta K).
\end{align*}
The first line has coefficient $\eta_n(Q_j^{K_n} - Q_j^K)$, which is $O_p(1/\sqrt{n})$ in sup norm (since $\eta_n$ is $O_p(1)$ and $Q_j^{K_n} - Q_j^K = O_p(1/\sqrt{n})$ by Proposition \ref{derivative_of_Ls_real_analytic} and the empirical process bound \cite[Theorem 5.8]{WatanabeGrey} applied to the normal derivatives of $f$).  The restricted covariance with this coefficient is therefore $O_p(1/\sqrt{n}) \to 0$.

The second line is the convergence of the functional restricted covariance with the deterministic coefficient $\eta Q_j^K$.  Since $Q_j^K$ is $x$-independent, the lift of $\eta Q_j^K$ is $Q_j^K(w) \wt\eta(x,w)$, which is $L^2(q)$-valued real analytic and inherits RFV from $\wt\eta$.  Theorem \ref{convergence_zeroth_order} applies, giving convergence to zero.

The subleading terms with $r_j < M$ converge to zero in probability.  Indeed, the prefactor $(n\beta)^{r_j - M} \to 0$, so it suffices to show the bracketed difference is bounded in probability.  Since $\eta_n$, $Q_j^{K_n}$, and $\Delta K_n$ are all bounded in probability on compact $T$ (by the empirical process bound \cite[Theorem 5.8]{WatanabeGrey}), each restricted covariance is of class $O_p(1)$.
\end{proof}

\begin{theorem} \label{asymptotic_unbiasedness_differential}
Under the hypotheses of Theorem \ref{convergence_differential}, suppose additionally that $\xi \in L^a(q)$ and that $f$ and all coefficient lifts of $P$ (and their derivatives up to order $M$) are $L^b(q)$-valued real analytic, where $\frac{1}{a} + \frac{1}{b} \le \frac{1}{4(M+1)}$.  Then
\[\bb E\Big[\mr{Cov}^{\mr{res}}_{\sPiEmp}(\OO_n, \Delta K_n) \Big] - \mr{Cov}^{\mr{res}}_{\sPiPop}(\OO, \Delta K) \to 0\]
as $n\beta \to \infty$ and $\beta \to 0$.
\end{theorem}

\begin{proof}
As in the proof of Theorem \ref{convergence_differential}, the restricted covariance decomposes into a finite sum of functional restricted covariances with coefficients $\rho_i \eta_{ij} Q_j^K$ (leading terms) and subleading terms suppressed by $(n\beta)^{r_j - M}$.  The H\"older condition ensures that each coefficient lift and $\xi f$ are $L^4(q)$-valued real analytic (using Proposition \ref{derivative_of_Ls_real_analytic} for the derivatives of $f$ appearing in $Q_j^K$).  Theorem \ref{asymptotic_unbiasedness} applies to each leading functional term, and the subleading terms vanish in expectation since they are $O_p((n\beta)^{r_j - M})$ with bounded second moments.
\end{proof}

\subsection{Susceptibility Estimation in Practice}
We conclude by summarizing the procedure by which one may define susceptibility estimators for arbitrary differential observables in practice.  As we discussed in Section \ref{SGLD_section} there is necessarily an additional layer of imprecision that one must layer on top of the estimation we have explained above, coming from the difficulty in directly sampling from the posterior distribution.  One must instead use an approximate sampling process like stochastic gradient Langevin dynamics, and this may introduce asymptotic biases that weren't present for the idealized estimators.

The results presented above lead to the following process.
\begin{enumerate}
    \item For each summand $P_i \delta_{S_i}$ of our observable $\OO$ find an atlas $\{U_{ij}\}$ for $W$ and subordinate partition of unity $\{\rho_{ij}\}$ for which we may write $P_i$ as a sum of terms $\partial^{\beta_{ijk}} \eta_{ijk} \delta_{T_{ijk}}$ in local normal form.
    \item Decompose each such summand into a sum of Leibniz terms $Q_{ijk\ell}^K \eta_{ijk} \delta_{T_{ijk}}$.  We may discard all Leibniz terms not of leading order.
    \item Let
    \[\OO_n = \sum_{ijk\ell} \rho_{ij} Q_{ijk\ell}^{K_n} \eta_{ijk} \delta_{T_{ijk}}\]
    be obtained by replacing the population level Leibniz coefficients by empirical Leibniz coefficients.
    \item Choose a chain length $T$ and a step size $\eps > 0$.  For each submanifold $T_{ijk}$ appearing in the decomposition, run an SGLD chain of length $T$ on $T_{ijk}$ to obtain approximate samples from $\Pi^{\mr{emp},T_{ijk}}_{n,\beta}$.  Run a separate chain on $W$ to obtain approximate samples from $\PiEmp$.
    \begin{definition}
        The \emph{SGLD estimator} for $\chi_{n,\beta}^{\mr{pop,ren}}(\OO, \xi)$ is
        \begin{multline*}
        \widehat\chi_{n,\beta}^{\mr{SGLD}}(\OO, \xi) = -\sum_{ijk} \frac{1}{(n\beta)^M} \sum_\ell (-1)^{|\beta_{ijk}|} (-n\beta)^{r_\ell} \Big(\widehat{\bb E}_{T_{ijk},n}^{\mr{SGLD}}\big(\rho_{ij}\, (\eta_{ijk})_n Q_{ijk\ell}^{K_n} \Delta K_n\big) \\
        - \widehat{\bb E}_{T_{ijk},n}^{\mr{SGLD}}\big(\rho_{ij}\, (\eta_{ijk})_n Q_{ijk\ell}^{K_n}\big) \cdot \widehat{\bb E}_{W,n}^{\mr{SGLD}}(\Delta K_n)\Big)
        \end{multline*}
        where $\widehat{\bb E}_{T_{ijk},n}^{\mr{SGLD}}$ denotes the sample mean over the SGLD chain on $T_{ijk}$ (after discarding burn-in) and $\widehat{\bb E}_{W,n}^{\mr{SGLD}}$ denotes the sample mean over the full-posterior chain.
    \end{definition}
\end{enumerate}

\begin{remark}
    Notice that this definition depends on many choices: not only on hyperparameters such as $T, \eps$, but also on the representation of our original observable as a sum of the form $\sum_i P_i \delta_{S_i}$ and on the choices of atlases and partitions of unity.  While Theorem \ref{asymptotic_unbiasedness_differential} tells us that the ideal parallels of these estimators (using genuine posterior samples) are asymptotically unbiased for each fixed representation, and therefore asymptotically independent of the atlas and partition of unity, this is not expected to be the case under SGLD sampling.
\end{remark}

\subsection*{Acknowledgements}
We are grateful to Zach Furman for sharing his forthcoming updates to \cite[Appendix A]{sing2}, whose arguments we adapted for the arguments of Section \ref{moment_lemma_section}.  We would also like to thank Rohan Hitchcock and Andy Gordon for their helpful comments on a previous version of this paper.

\appendix

\section{Pseudoinverses of Susceptibility Matrices} \label{pseudoinverse}

The technique of \emph{patterning} \cite{patterning} uses susceptibilities to engineer a perturbation in the space of distributions that induces a desired effect on the model, as detected by the posterior expectation values of a specified set of observables.  This involves the calculation of an \emph{inverse} to a linear map whose matrix entries are defined by susceptibilities.  In this appendix we investigate the application of our consistency and asymptotic unbiasedness results from \S\ref{sec:functional}--\S\ref{sec:differential} to define empirical estimators applicable to this inverse problem.

We begin by setting up the problem at hand.  Fix observables $\OO_1, \ldots, \OO_H$ and perturbations $\xi_1, \ldots, \xi_m \in L^4_0(q) \cap C^\infty(X)$ each satisfying the hypotheses of Theorem \ref{convergence_differential}.  The perturbations span a finite-dimensional subspace $V \sub T_q \Delta^+(X)$ of the tangent space at the true distribution.  

\begin{definition}
The (population) \emph{susceptibility operator} is the linear map
\[\chi^{\mr{pop}}_{n,\beta} \colon V \to \RR^H, \qquad (\chi^{\mr{pop}}_{n,\beta})_{ij} = \chi^{\mr{pop,ren}}_{n,\beta}(\OO_i, \xi_j).\]
\end{definition}

We will identify the susceptibility operator with an element in $\RR^{H \times m}$ using the bases $\OO_i$ and $\xi_j$.  We will sometimes write $\chi^{\mr{pop}}$ when $n,\beta$ are clear from context.  The empirical counterpart $\chi^{\mr{emp}}_{n,\beta} \in \RR^{H \times m}$ has entries $(\chi^{\mr{emp}}_{n,\beta})_{ij} = \widehat\chi^{\mr{ren}}_{n,\beta}(\OO_i, \xi_j)$.  By Theorem \ref{convergence_differential} we have $\chi^{\mr{emp}}_{n,\beta} - \chi^{\mr{pop}}_{n,\beta} \to 0$ in probability entrywise as $n\beta \to \infty$, $\beta \to 0$, and under the hypotheses of Theorem \ref{asymptotic_unbiasedness_differential} the convergence also holds in expectation.

Given a target vector $b \in \RR^H$ specifying a desired first-order change in posterior expectations, the Moore--Penrose pseudoinverse provides the minimum-norm least-squares solution to $\chi^{\mr{pop}}_{n,\beta} h = b$ \cite{penrose1956}. 

\begin{definition}
We call the tangent vector \[h^{\mr{pop}}_{n,\beta} = (\chi^{\mr{pop}}_{n,\beta})^+ b \in V\]
the \emph{patterning perturbation} associated to $(\OO_i)_i$, $(\xi_j)_j$, and $b$.
\end{definition}

We've defined the patterning perturbation at the population level; we would like to use empirical susceptibilities to define estimators for the patterning perturbation at finite $n$.  The most na\"ive approach would be to replace $\chi^{\mr{pop}}_{n,\beta}$ by $\chi^{\mr{emp}}_{n,\beta}$ in the formula above.  This na\"ive estimator is not asymptotically unbiased in general.  The Moore--Penrose pseudoinverse is a discontinuous function of its argument: it is continuous on each rank stratum in $\RR^{H \times m}$ but has jump discontinuities across the boundaries between strata.  When $\chi^{\mr{pop}}_{n,\beta}$ is not full rank, generic perturbations $\chi^{\mr{emp}}_{n,\beta} = \chi^{\mr{pop}}_{n,\beta} + E$ increase the rank, introducing spurious small singular values whose reciprocals dominate $(\chi^{\mr{emp}}_{n,\beta})^+$.  We address this by introducing ridge regularization.

\begin{definition} \label{def:ridge_inverse}
Fix $\lambda > 0$.  The \emph{ridge-regularized inverse} is the function
\[R_\lambda \colon \RR^{H \times m} \to \RR^{m \times H}, \qquad R_\lambda(A) = A^T(AA^T + \lambda I_H)^{-1}.\]
The \emph{$\lambda$-regularized patterning perturbation} is $h^{\mr{pop},\lambda}_{n,\beta} = R_\lambda(\chi^{\mr{pop}}_{n,\beta}) b$, and its empirical estimator is $\widehat h_{n,\beta}^\lambda = R_\lambda(\chi^{\mr{emp}}_{n,\beta}) b$.
\end{definition}

The matrix $AA^T + \lambda I_H$ is positive definite for every $A \in \RR^{H \times m}$, so $R_\lambda$ is defined on all of $\RR^{H \times m}$.  In the singular value decomposition basis, $R_\lambda$ replaces the diagonal of $A^+$ by $\sigma/(\sigma^2 + \lambda)$, so as $\lambda \to 0$ the ridge-regularized inverse converges elementwise to the Moore--Penrose pseudoinverse.

\begin{theorem} \label{ridge_consistency}
Fix $\lambda > 0$.  As $n\beta \to \infty$ and $\beta \to 0$:
\begin{enumerate}
\item Under the hypotheses of Theorem \ref{convergence_differential}, the empirical $\lambda$-regularized patterning perturbation is consistent:
\[\widehat h_{n,\beta}^\lambda - h^{\mr{pop},\lambda}_{n,\beta} \to 0\]
in probability.
\item Under the additional hypotheses of Theorem \ref{asymptotic_unbiasedness_differential}, the estimator is asymptotically unbiased:
\[\bb E[\widehat h_{n,\beta}^\lambda] - h^{\mr{pop},\lambda}_{n,\beta} \to 0.\]
\end{enumerate}
\end{theorem}

\begin{proof}
By Theorem \ref{convergence_differential}, $\chi^{\mr{emp}}_{n,\beta} - \chi^{\mr{pop}}_{n,\beta} \to 0$ in probability entrywise; since $H$ and $m$ are finite this is equivalent to convergence in operator norm.  The map $A \mapsto R_\lambda(A)$ is a rational function on $\RR^{H \times m}$ by construction, and since $\det(AA^T + \lambda I_H)$ is non-zero for all $A$ it is real analytic.  Therefore we can apply the continuous mapping theorem, which tells us that $R_\lambda(\chi^{\mr{emp}}_{n,\beta}) - R_\lambda(\chi^{\mr{pop}}_{n,\beta}) \to 0$ in probability.  Applying both sides to the vector $b$ proves the first claim.

For the second claim it suffices to show that the family $\{\widehat h_{n,\beta}^\lambda\}$ is uniformly integrable as $n \to \infty$, since convergence in probability together with uniform integrability implies convergence in $L^1$ by Vitali's convergence theorem.  We use the deterministic operator norm bound
\[\|R_\lambda(A)\|_{\mr{op}} \le \frac{1}{2\sqrt\lambda},\]
which holds for every $A \in \RR^{H \times m}$ and every $\lambda > 0$ \cite[Theorems 2.6 and 2.8(a)]{Kirsch}.  Hence $\|\widehat h_{n,\beta}^\lambda\| \le \|b\|/(2\sqrt\lambda)$ almost surely, which makes the family $\{\widehat h_{n,\beta}^\lambda\}$ bounded in $L^\infty$ and thus uniformly integrable.  
\end{proof}

\begin{remark} \label{rmk:vanishing_lambda}
The theorem holds for any fixed $\lambda > 0$, in which case the limiting object $h^{\mr{pop},\lambda}_{n,\beta}$ is a \emph{regularized} version of the patterning perturbation $h^{\mr{pop}}_{n,\beta}$.  In the SVD basis, $R_\lambda(\chi^{\mr{pop}}_{n,\beta})$ acts diagonally with non-zero entries $\sigma_\alpha/(\sigma_\alpha^2 + \lambda)$, giving an approximation error
\[\|R_\lambda(\chi^{\mr{pop}}_{n,\beta}) - (\chi^{\mr{pop}}_{n,\beta})^+\|_{\mr{op}} = \frac{\lambda}{\sigma_{\mr{min}}(\sigma_{\mr{min}}^2 + \lambda)}\]
as $\lambda \to 0$, where $\sigma_{\mr{min}}$ is the smallest positive singular value of $\chi^{\mr{pop}}_{n,\beta}$.  In particular, while one could study the behavior of our estimators as $\lambda \to 0$ together with $n\beta \to \infty$, this would require a uniform lower bound on $\sigma_{\mr{min}}(\chi^{\mr{pop}}_{n,\beta})$ for $n\beta$ sufficiently large, which we do not address here.
\end{remark}

\bibliographystyle{alpha}
\bibliography{references}

\newcommand{\etalchar}[1]{$^{#1}$}
\begin{thebibliography}{WHvW{\etalchar{+}}25}

\bibitem[AFH25]{sing2}
Maxwell Adam, Zach Furman, and Jesse Hoogland.
\newblock The loss kernel: A geometric probe for deep learning interpretability.
\newblock {\em arXiv:2509.26537}, 2025.

\bibitem[BG14]{bochkina2014nonregular}
Natalia~A. Bochkina and Peter~J. Green.
\newblock The {B}ernstein--von {M}ises theorem and nonregular models.
\newblock {\em The Annals of Statistics}, 42(5):1850--1878, 2014.

\bibitem[BWHM25]{lang2}
George Baker, George Wang, Jesse Hoogland, and Daniel Murfet.
\newblock Structural inference: Interpreting small language models with susceptibilities.
\newblock {\em arXiv:2504.18274}, 2025.

\bibitem[DP17]{drton2017sbic}
Mathias Drton and Martyn Plummer.
\newblock A {B}ayesian information criterion for singular models.
\newblock {\em Journal of the Royal Statistical Society: Series B (Statistical Methodology)}, 79(2):323--380, 2017.

\bibitem[GBJ18]{giordano2018covariances}
Ryan Giordano, Tamara Broderick, and Michael~I. Jordan.
\newblock Covariances, robustness, and variational {B}ayes.
\newblock {\em Journal of Machine Learning Research}, 19(51):1--49, 2018.

\bibitem[GBW{\etalchar{+}}26]{lang3}
Andrew Gordon, Garrett Baker, George Wang, William Snell, Stan van Wingerden, and Daniel Murfet.
\newblock Towards spectroscopy: Susceptibility clusters in language models, 2026.

\bibitem[Gus00]{gustafson2000local}
Paul Gustafson.
\newblock Local robustness in {B}ayesian analysis.
\newblock In David Rios~Insua and Fabrizio Ruggeri, editors, {\em Robust {B}ayesian Analysis}, volume 152 of {\em Lecture Notes in Statistics}, pages 71--88. Springer, New York, 2000.

\bibitem[Ham74]{hampel1974influence}
Frank~R. Hampel.
\newblock The influence curve and its role in robust estimation.
\newblock {\em Journal of the American Statistical Association}, 69(346):383--393, 1974.

\bibitem[HH25]{hitchcock2025global}
Rohan Hitchcock and Jesse Hoogland.
\newblock From global to local: A scalable benchmark for local posterior sampling, 2025.

\bibitem[Hit26]{HitchcockThesis}
Rohan Hitchcock.
\newblock {\em Singular learning theory for deep learning interpretability}.
\newblock PhD thesis, University of Melbourne, 2026.

\bibitem[Joy12]{Joyce}
Dominic Joyce.
\newblock On manifolds with corners.
\newblock In Stanis{\l}aw Janeczko, Jun Li, and Duong~H. Phong, editors, {\em Advances in Geometric Analysis}, volume~21 of {\em Advanced Lectures in Mathematics}, pages 225--258. International Press, 2012.
\newblock arXiv:0910.3518.

\bibitem[Kir21]{Kirsch}
Andreas Kirsch.
\newblock {\em An Introduction to the Mathematical Theory of Inverse Problems}, volume 120 of {\em Applied Mathematical Sciences}.
\newblock Springer, 3rd edition, 2021.

\bibitem[KL17]{koh2017understanding}
Pang~Wei Koh and Percy Liang.
\newblock Understanding black-box predictions via influence functions.
\newblock In {\em Proceedings of the 34th International Conference on Machine Learning (ICML)}, volume~70 of {\em PMLR}, pages 1885--1894, 2017.

\bibitem[KM97]{KrieglMichor}
Andreas Kriegl and Peter~W. Michor.
\newblock {\em The Convenient Setting of Global Analysis}, volume~53 of {\em Mathematical Surveys and Monographs}.
\newblock American Mathematical Society, 1997.

\bibitem[Kub57]{Kubo}
Ryogo Kubo.
\newblock Statistical-mechanical theory of irreversible processes. i. general theory and simple applications to magnetic and conduction problems.
\newblock {\em Journal of the Physical Society of Japan}, 12(6):570--586, 1957.

\bibitem[KWA{\etalchar{+}}25]{singfluence1}
Philipp~Alexander Kreer, Wilson Wu, Maxwell Adam, Zach Furman, and Jesse Hoogland.
\newblock Bayesian influence functions for hessian-free data attribution, 2025.

\bibitem[Mel96]{Melrose}
Richard~B Melrose.
\newblock {\em Differential Analysis on Manifolds with Corners}.
\newblock \url{http://www-math.mit.edu/~rbm/book.html}, 1996.

\bibitem[MRP{\etalchar{+}}25]{mlodozeniec2025distributional}
Bruno Mlodozeniec, Isaac Reid, Sam Power, David Krueger, Murat~A. Erdogdu, Richard~E. Turner, and Roger Grosse.
\newblock Distributional training data attribution: What do influence functions sample?
\newblock In {\em Advances in Neural Information Processing Systems 38 (NeurIPS)}, 2025.

\bibitem[Pen56]{penrose1956}
Roger Penrose.
\newblock On best approximate solutions of linear matrix equations.
\newblock {\em Mathematical Proceedings of the Cambridge Philosophical Society}, 52(1):17--19, 1956.

\bibitem[RAMR26]{ray2026tempering}
Ruchira Ray, Marco Avella~Medina, and Cynthia Rush.
\newblock Statistical guarantees for data-driven posterior tempering.
\newblock {\em arXiv preprint arXiv:2601.09122}, 2026.

\bibitem[VZT16]{vollmer2016sgld}
Sebastian~J. Vollmer, Konstantinos~C. Zygalakis, and Yee~Whye Teh.
\newblock Exploration of the {(Non-)A}symptotic bias and variance of stochastic gradient {L}angevin dynamics.
\newblock {\em Journal of Machine Learning Research}, 17(159):1--48, 2016.

\bibitem[Wat09]{WatanabeGrey}
Sumio Watanabe.
\newblock {\em Algebraic Geometry and Statistical Learning Theory}, volume~25 of {\em Cambridge Monographs on Applied and Computational Mathematics}.
\newblock Cambridge University Press, 2009.

\bibitem[Wat13]{WatanabeWBIC}
Sumio Watanabe.
\newblock A widely applicable {B}ayesian information criterion.
\newblock {\em Journal of Machine Learning Research}, 14(27):867--897, 2013.

\bibitem[Wat18]{WatanabeGreen}
Sumio Watanabe.
\newblock {\em Mathematical Theory of {B}ayesian Statistics}.
\newblock CRC Press, 2018.

\bibitem[WHvW{\etalchar{+}}25]{lang1}
George Wang, Jesse Hoogland, Stan van Wingerden, Zach Furman, and Daniel Murfet.
\newblock Differentiation and specialization of attention heads via the refined local learning coefficient.
\newblock In {\em Proceedings of The 13th International Conference on Learning Representations}, 2025.

\bibitem[WM26]{patterning}
George Wang and Daniel Murfet.
\newblock Patterning: The dual of interpretability, 2026.

\bibitem[WT11]{wellingteh2011}
M.~Welling and Y.~W. Teh.
\newblock Bayesian learning via stochastic gradient {L}angevin dynamics.
\newblock In {\em Proceedings of the 28th International Conference on Machine Learning}, 2011.

\end{thebibliography}

\end{document}